\documentclass[leqno,11pt]{amsart}

\usepackage{palatino}
\usepackage[mathcal]{euler}

\usepackage{xcolor}

\definecolor{DarkBlue}{rgb}{0.2,0.2,0.4}

\usepackage{graphicx}
\usepackage{epstopdf,epsfig}
\usepackage{amsmath}
\usepackage{pst-node}
\usepackage{tikz-cd} 
\usepackage{comment}
\usepackage{xy}
\xyoption{all}
\usepackage{palatino}
\usepackage{tikz}
\usepackage{comment}
\usepackage[mathcal]{euler}
\usepackage{extarrows}
\usepackage{amsmath,amsthm,amsfonts,amssymb,latexsym,amscd,enumerate}
\usepackage{extarrows}
\usepackage{mathtools}

\usepackage{color}
\usepackage{slashed}

\theoremstyle{plain}
    \newtheorem{theorem}[equation]{Theorem}
    \newtheorem{lemma}[equation]{Lemma}
    \newtheorem{corollary}[equation]{Corollary}

    \newtheorem*{theorem*}{Theorem}
    
	\newtheorem{construction}[equation]{Construction}

 \theoremstyle{definition}
    \newtheorem{definition}[equation]{Definition}
    
    \newtheorem{example}[equation]{Example}

    \newtheorem{remark}[equation]{Remark}
        \newtheorem*{remark*}{Remark}

\theoremstyle{remark}

\numberwithin{equation}{section}

\usepackage[colorlinks=true, linkcolor=DarkBlue, citecolor=DarkBlue]{hyperref}

\newcommand{\R}{{\mathbb{R}}}
\newcommand{\Z}{{\mathbb{Z}}}
\newcommand{\C}{{\mathbb{C}}}

\newcommand{\Sc}{\mathcal{S}}
\newcommand{\Uc}{\mathcal{U}}

\newcommand{\gf}{\mathfrak{g}}

\newcommand{\len}{{\rm len}}

\DeclareMathOperator{\lcm}{lcm}

\DeclareMathOperator{\dist}{dist}

\numberwithin{equation}{section}

\begin{document}

\title{Uniformity of Maximal Hypoellipticity on Graded Lie Groups: From Pointwise to Global} 

\author{Shiqi Liu} 
\author{Edward McDonald}
\author{Fedor Sukochev}
\author{Dmitriy Zanin} 
 
\dedicatory{Dedicated to Nigel Higson}

\begin{abstract}
	On graded Lie groups, we develop a mechanism that transfers the uniformity of maximal hypoellipcity from the frozen coefficients principal part of a differential operator to the full operator.
	Our approach brings the century-old "freeze-unfreeze" strategy into the hypoelliptic setting, and	 offers a transparent and flexible framework for lifting symbol-level hypoelliptic properties to global elliptic estimates, without relying on pseudodifferential calculus.
	In addition, we prove that symmetric operators of hypoelliptic type on a graded Lie group are self-adjoint.
\end{abstract}

\maketitle

\tableofcontents

\section{Introduction}\label{sec-Intro}

\subsection{Overview}

Hypoelliptic operators occupy a significant position in modern harmonic analysis and in the theory of partial differential equations. In 1967,  H\"{o}rmander's seminal work \cite{Hormander1967} provided an algebraic sufficient condition for hypoellipticity of second order operators built from vector fields. In 1976, Rothschild and Stein  \cite{RothschildStein1976} introduced a lifting and approximation scheme which reduces local questions for H\"{o}rmander systems on manifolds to model problems on nilpotent Lie groups. Along with other pioneering works of the same era \cite{Folland_Stein_complex_Heisenberg1974, Folland_Applications_Nilpotent_to_PDE_1977, FollandStein1982, HelfferNourrigat1985, Rockland1978}, the door was opened to the use of representation theoretic methods in analysis and complex geometry.

Among these developments, maximally hypoelliptic operators were singled out by Helffer and Nourrigat \cite{HelfferNourrigat1985}, as they obey good \emph{a priori} estimates. On nilpotent Lie groups, Helffer and Nourrigat  established a  representation theoretic criterion for maximal hypoellipticity, resolving a conjecture of Rockland \cite{Rockland1978}. They also conjectured an analogous result for general filtered manifolds. Recently, Helffer and Nourrigat's conjecture was resolved by Androulidakis, Mohsen, and Yuncken \cite{AMY-Helffer-Nourrigat-conjecture} through a  pseudodifferential calculus adapted to a filtered manifold. Their remarkable work established that injectivity  of an adapted principal symbol is equivalent to maximal hypoellipticity.

In this paper, we establish a pointwise-to-global principle for differential operators on graded Lie groups (Theorem \ref{gee_introduction} and \ref{general elliptic estimate theorem 2}): if the "frozen" principal part of such an operator is maximally hypoelliptic with uniform constants, then the operator satisfies global \emph{a priori} estimates, improving regularity by its full order with respect to appropriately defined Sobolev scales, which is a global version of maximal hypoellipticity. Further, we show that uniformly maximally hypoelliptic symmetric operators on graded Lie groups are self-adjoint (Theorem \ref{elliptic regularity theorem}).

The underlying idea is essentially the classic ``freeze-unfreeze" strategy which has a long history (see, e.g. \cite{Schauder1934}).
Concretely, we first obtain estimates for right invariant (constant coefficient) model operators on the graded group; next, we allow coefficients to vary within small open sets and prove stability of  local estimates; finally, a partition of unity patches these into global bounds.

Our strategy is inspired by classical ideas (see, e.g. \cite{Zimmer1990}) in the Euclidean setting, while the nature of the graded context makes our approach different at several fundamental points.
For instance, for symmetric translation invariant differential operators, hypoellipticity implies the dense-range property \cite[Proposition 4.1.15.]{FischerRuzhansky2016}, but it is not straightforward to generalizing this result to non translation invariant operators. In our approach, the dense-range property emerges naturally as a consequence of the forward (Theorem \ref{forward elliptic estimate}) and backward (Theorem \ref{backward elliptic estimate}) estimates, rather than serving as an intermediate step.

Hypoellipticity and maximal hypoellipticity are defined as local properties, while our work focuses on a global form of maximal hypoellipticity. We transfer maximal hypoellipticity from the pointwise principal part to the full operator. We hope that this viewpoint provides an additional perspective to the classical works of Rothschild, Stein, Helffer, and Nourrigat \cite{RothschildStein1976, Rothschild_Criterion_1979}, \cite{HelfferNourrigat1985}, the more recent innovative work of van Erp and Yuncken \cite{vanErpYuncken2019}, and profound geometric applications of hypoelliptic operators due to Bismut, Shen, and Wei \cite{Bismut-Orbit-Int2011, BismutShuWei-RRG23}

The established analysis of such operators proceeds through constructing a pseudodifferential calculus, building a parametrix, and finally obtaining estimates,  for instance, \cite{BealsGreiner1988}, \cite{ChristGellerGlowacki1992}, \cite{Goodman-lnm-1976}, \cite{FischerRuzhansky2016} and \cite{Street-maximal-subellipticity-2023}. During the preparation of this paper, we became aware of the work of Fermanian-Kammerer, Fischer, and Flynn \cite{Fischer2024calculus_filtered_manifolds}. We believe that their framework could also be adapted to obtain similar results  through this traditional route.

\subsection{Graded Lie groups}\label{sec-intro-graded-group}

A {graded} Lie group $G$ is a connected and simply connected Lie group whose Lie algebra $\mathfrak{g}$ is finite dimensional, equipped with a vector space decomposition $\mathfrak{g} = \bigoplus_{k=1}^\infty V_k$, satisfying {$[V_k,V_j] \subseteq V_{k+j}$} (for reference, see  \cite[Definition 3.1.1.]{FischerRuzhansky2016}). We ask $\gf$ to be finite dimensional throughout the paper, then the direct sum is finite, thus $\gf$ is nilpotent, that is, $[\gf,[...,[\gf, \gf]]]=0$, for finitely many steps of taking Lie brackets. The homogeneous dimension $d_{\hom}$ is defined as
{
\[
d_{\hom} = \sum_{k=1}^\infty k\cdot \mathrm{dim}(V_k).
\]
Notice that $d_{\hom}$ highly depends on the grading.
}
Since $G$ is connected, simply connected and nilpotent, it follows that the exponential map $\exp:\gf\to G$ is a diffeomorphism, and so we may assume that $G$ and $\gf$ coincide as sets. With this identification, the Lebesgue measure on $\gf$ is a bivariant Haar measure for $G$ \cite[Proposition 1.2]{FollandStein1982}. The Haar measure on $G$ is therefore written as $dx,$ so that $\int_{G} f(x)\,dx$ means the integral of a measurable function $f$ on $G$ with respect to the Haar measure of $G.$ The elements of $\gf$ can be identified with right-invariant derivations of compactly supported smooth functions on $G$ ({denote such function class as $C^\infty_c(G)$}), if we identify $X\in \gf$ as the generator of the group of {\emph{left} translations
    \[
        Xu(g) := \frac{d}{dt}u(\exp(-tX)g)|_{t=0},\quad g \in G,\, u\in C^\infty_c(G).
    \]
    }

We will fix a choice of a preferred generating set $\{X_j\}_{j=1}^{n'}$ (see Definition \ref{def-generator-basis}) for the remainder of this paper. We will denote $\{v_j\}_{j=1}^{n'}$ for their degrees, and their least common multiple $v = \lcm(\{v_j\}_{j=1}^{n'})$.
Similar to \cite[Corollary 4.1.10.]{FischerRuzhansky2016}, we define an associated operator $\Delta_{G}$:
   \begin{equation}\label{eq-def-new-Laplacian}
       \Delta_{G} = - \sum_{j=1}^{n'} (-1)^{\frac{v}{v_j}}X_j^{\frac{2v}{v_j}}.
   \end{equation}

Mostly we will write $\Delta$ for $\Delta_G$ without ambiguity. 
{The operator $\Delta_G$ is a homogeneous differential operator of order $2v$.} It is typically not elliptic, but hypoelliptic in the classical sense, as in e.g. \cite[Section 11.1]{Hormander-II}.
{A graded Lie algebra $\gf$ is said to be stratified if $\gf_1$ generates $\gf$.
For stratified $G$, that is $\gf$ being stratified, the set of preferred generators can be selected to consist entirely of homogeneous elements of order 1. Thus, $v=1$ and $\Delta_G$ becomes the sub-Laplacian.

\subsection{Differential operators on graded Lie groups}\label{sec-intro-differential-graded}

We define differential operators on graded Lie groups, in analogy with those in the Euclidean case. The stratified case was long-established, for example, in \cite{Folland1975}. For the general graded case, see \cite{FischerRuzhansky2016}.  Given a word $\alpha = \alpha_1\ldots\alpha_k$
in the alphabet $\alpha_i \in \{1,2,\ldots,n'\},$ write
\begin{equation}\label{eq-def-diff-op}
    X^{\alpha} = X_{\alpha_1}X_{\alpha_2}\cdots X_{\alpha_k}.
\end{equation}
{
Recall that $\{X_i\}_{j=1}^{n'}$, with degrees $\{v_j\}_{j=1}^{n'}$, is a fixed choice of preferred generators (Definition \ref{def-generator-basis}).
Additionally, we define
\begin{equation}\label{eq-def-len}
    \mathrm{len}(\alpha)= \sum_{j=1}^{k}\deg(X_{\alpha_j})=\sum_{j=1}^{k} v_{\alpha_j}.
\end{equation}
Note that $\mathrm{len}(\alpha)$ is not the usual length $k$ of the word $\alpha,$ but is weighted by the degrees $v_j.$ In the literature, $\len(\alpha)$ has been called a \textit{weighted length} \cite[Page 2]{terElstRobinson97}.
}

\begin{definition} A differential operator of order at most $m$ on $G$ is a linear operator given by
    \begin{equation}\label{eq-P-expression}
        P = \sum_{\len(\alpha)\leqslant m} M_{a_{\alpha}}X^{\alpha}:\Sc(G)\to \Sc(G).
    \end{equation}
    Here, every $a_{\alpha}$ is a smooth function on $G$, and $\Sc(G)$ is the set of Schwartz class functions on $G$. We only consider operators $P$ with every $a_{\alpha}$ being uniformly smooth in the correct sense (see Section \ref{subsec-3.2} below).

    For operator $P$ with expression \eqref{eq-P-expression}, we denote by $P^\dagger$ its formal adjoint given by:
    \begin{equation}\label{eq-P-dagg-expre}
        P^\dagger = \sum_{\len(\alpha)\leqslant m} (X^{\alpha})^\dagger M_{\overline{a_{\alpha}}} :\Sc(G)\to \Sc(G) ,
    \end{equation}
    where if $\alpha=\alpha_1\cdots\alpha_k,$ we have $(X^{\alpha})^\dagger = (-1)^kX_{\alpha_k}\cdots X_{\alpha_1}.$ 

    Let $\widetilde{P}:\Sc'(G)\to\Sc'(G)$ denote the extension of $P$ to distributions, defined on $\omega \in \Sc'(G)$ by
    \begin{equation}\label{eq-P-tilde-expre}
        \langle \widetilde{P}\omega,\phi\rangle = \langle \omega,P^\dagger\phi\rangle,\quad \phi\in \Sc(G).
    \end{equation}
\end{definition}

\begin{remark}
    {
    Since the monomials $X^{\alpha}$ are not linearly independent, the expression \eqref{eq-P-expression} for $P$ might not be unique, so it is not immediate that $P^{\dagger}$ is well-defined by \eqref{eq-P-dagg-expre}. However $P^\dagger$ is related to $P$ by the adjoint relation
    \begin{equation*}
        \langle P\omega,\phi\rangle = \langle \omega,P^\dagger\phi\rangle,\qquad \omega, \phi\in \Sc(G).
    \end{equation*}
    Hence, as an operator $P^\dagger$ is uniquely determined by the operator $P.$}
\end{remark}

If $P$ has a representation where {all of the coefficient functions $a_{\alpha}$ are constants}, we say that $P$ is a constant coefficient differential operator. {In this case, $P$ may be identified with an element of the universal enveloping algebra $\Uc(\gf)$. For more background on the universal enveloping algebra, see Section \ref{sec-universal-enve-alg}.}

As a replacement for the notion of principal symbol, we consider the constant coefficient operator obtained from $P$ as follows.
\begin{definition} 
Let $P$ be a differential operator as in \eqref{eq-P-expression}. For $g \in G,$ denote
    \[
        P_g^{{\rm top}}=\sum_{\len(\alpha)=m}a_{\alpha}(g)X^{\alpha}.
    \]
{Since each $a_{\alpha}(g)$ is constant, thus $P_g^{\rm top} \in \Uc(\gf)$.}
\end{definition}

The operator $P_g^{{\rm top}}$ is the analogy of the Fourier transform of the principal symbol in pseudodifferential calculus.

In the nilpotent group setting, several different notions of ellipticity have been proposed, such as the maximal hypoellipticity of Helffer and Nourrigat \cite[Definition 1.1]{HelfferNourrigat1985} and the Rockland condition \cite{Rockland1978}. For our purposes the following notion is most useful, see also Definition \ref{definition of ellipticity_1} below.
\begin{definition}\label{definition of ellipticity} Let $P$ be a differential operator of order $m.$ We say that $P$ is uniformly Rockland if there exists a constant $c_P>0$ such that
$$\|P_g^{{\rm top}}u\|_{L_2(G)}\geqslant c_P\|(-\Delta_G)^{\frac{m}{2{v}}} u\|_{L_2(G)},\quad u\in\Sc(G),\quad g\in G.$$
\end{definition}

\begin{remark}\label{rmk-elliptic-representation-verify}
Definition \ref{definition of ellipticity} can be stated in terms of representation theory. A differential operator $P$ of order $m$ is  uniformly Rockland  in the sense of Definition \ref{definition of ellipticity} if and only if there exists a constant $c_P$ such that for all unitary irreducible representations $(\pi,H_{\pi})$ of $G,$ we have
$$\|\pi(P_g^{{\rm top}})\eta\|_{H_{\pi}} \geqslant c_P\|\pi((-\Delta_G)^{\frac{m}{2{v}}})\eta\|_{H_{\pi}},\quad \eta \in H^{\infty}_{\pi}.$$
Here, $H^{\infty}_{\pi}$ denotes the space of all smooth vectors in $H_{\pi},$ i.e. the joint domain of $\pi(X^{\alpha})$ over all words $\alpha.$ {Notice that the constant $c_P$ does not depend on the representation $\pi$}. This equivalence is illustrated in Appendix \ref{higson_appendix}, specifically, see Theorem \ref{thm-abstract-rockland-general-g}.
\end{remark}

Also associated to the {grading} of $G$ is a canonical scale of Sobolev spaces $\{W^s_2(G)\}_{s\in \R}.$ We defer their definition to Section \ref{sec-3}.

\subsection{Main Results}
Our first theorem shows that a symmetric uniformly maximally hypoelliptic  operator is self-adjoint and improves regularity. In the Euclidean setting, this is a classical result that could be found in textbooks, for example Theorem 6.3.12 and 6.3.14 in \cite{Zimmer1990}.
\begin{theorem}\label{elliptic regularity theorem}
Let $P=P^{\dagger}$ be an  uniformly Rockland  order $m$ differential operator on $G.$
\begin{enumerate}[{\rm (i)}]
\item\label{erta} (Elliptic regularity) If $u\in L_2(G)$ is such that $\widetilde{P}u\in L_2(G),$ then $u\in W^m_2(G).$
\item\label{ertb} (Self-adjointness) {$\widetilde{P}$ is a self-adjoint operator on $L_2(G)$ with domain $W^m_2(G).$}
\end{enumerate}
\end{theorem}
The following theorem shows uniformly Rockland operators obey \emph{a priori} estimates.
\begin{theorem}\label{gee_introduction}
    Let $P=P^\dagger$ be an  uniformly Rockland  order $m$ differential operator on $G.$ Then for every $s \in \mathbb{R},$ there exists $c_{P,s}>0$ such that for all real $c$ with $|c|$ sufficiently large, we have
    \[
        \|u\|_{W^{s+m}_2(G)} \leqslant c_{P,s}\|(P+ic)u\|_{W^s_2(G)},\quad u \in \Sc(G).
    \]
\end{theorem}

Further, we remove the condition $P=P^{\dagger}$.
\begin{theorem}\label{general elliptic estimate theorem 2} 
Let $P$ be a differential operator of order $m$. If $P$ is  uniformly Rockland, then for every $s\in \mathbb{R}$, there exist constants $c_{P,s,1}, c_{P,s,2}>0$ such that 
\begin{equation*}
	c_{P,s,1}\|u\|_{W^{s+m}_2(G)}\leqslant \|Pu\|_{W^s_2(G)} + c_{P,s,2}\|u\|_{L_2(G)} ,\quad u\in \Sc(G).
\end{equation*}
\end{theorem} 

The proof of Theorem \ref{gee_introduction} in the case $s=0$ contains the main technical details and occupies a substantial portion of Section \ref{sec-4}.}
It is based on approximating $P$ by a constant coefficient operator, similar to the proof that elliptic operators obey elliptic estimates as in e.g. \cite[Chapter 6]{Zimmer1990} or \cite[Theorem 15.1]{AgmonDouglisNirenberg1959}.

For comparison, as defined qualitatively, hypoellipticity does not, by itself, guarantee improvement of regularity in Sobolev norms. Indeed, a notable counterexample was built by Kohn \cite{Kohn05_los}. On the other hand,  as defined quantitatively, maximally hypoelliptic operators are those operators $P$ that satisfy, for any bounded open subset $U\subset G$,
\begin{equation*} 
	\sum_{|\alpha|\leqslant m} \Vert X^{\alpha}u \Vert_{L_2(G)}\leqslant C_U (\Vert P u\Vert_{L_2(G)}+ \Vert u\Vert_{L_2(G)}), \quad u\in C^{\infty}_c(U).
\end{equation*}

Our statements share the same spirit with the literature.
For second order differential operators on a $CR$ manifold, Beals and Greiner established \emph{a priori}  estimates from conditions imposed on the "frozen" principal part of the operator, see Theorem 2.9, Theorem 18.4, and (18.33) in \cite{BealsGreiner1988}.
For general filtered manifolds, analogous estimates follow from a representation theoretical condition \cite{AMY-Helffer-Nourrigat-conjecture}.

\subsection{Structure of the paper}
\begin{itemize}
    \item In Section \ref{preliminaries}, we recall background materials about {graded} Lie groups.
    \item In Section \ref{sec-3}, we build a special partition of unity which plays an important role in the proofs of the main theorems. Using this partition of unity, we also localize the Sobolev norm on a {graded} Lie group (Theorem \ref{sobolev localization theorem}).
    \item In Section \ref{sec-4}, we prove the "elliptic" ( uniformly Rockland ) operators defined by Definition \ref{definition of ellipticity} satisfy elliptic estimates, see Theorem \ref{gee_introduction} and \ref{general elliptic estimate theorem}.
    \item In \ref{higson_appendix}, we give a proof of a theorem of Nigel Higson about verifying ellipticity via representations,  see Remark \ref{rmk-elliptic-representation-verify}.
\end{itemize}


\section{Preliminaries}\label{preliminaries}

\subsection{Graded Lie group}

{
\begin{definition}\label{def-stratfied-g}
    A Lie algebra $\gf$ is said to be graded if equipped with a direct sum vector space decomposition
    \begin{equation}\label{eq-stratfied-decomp}
        \gf  = \bigoplus_{i=1}^\infty V_i
    \end{equation}
    such that
    \begin{equation}\label{eq-stratified-relation}
        [V_j, V_k] \subseteq V_{j+k}, \quad j,  k \geqslant 1,
    \end{equation}
    In this paper, we only consider finite dimensional Lie algebras, thus there exists a smallest integer $s\geqslant 1$, such that $V_{i}=\{0\}$ for all $i> s$. Such $s$ is the highest degree of homogeneous elements.
    
    A Lie group $G$ is called graded, if it is connected, simply connected, and its Lie algebra $\gf$ is graded.
\end{definition}

\begin{definition}
    A Lie algebra $\gf$ is nilpotent if for
    \begin{equation*}
        \gf_0 \coloneqq \gf, \quad \gf_{j} \coloneqq [\gf, \gf_{j-1}], \ j\geqslant 1,
    \end{equation*}
    there exists an integer $s\geqslant 1$, such that $\gf_j = \{0\}$, $\forall j\geqslant s$. The minimum of such $s$ is called the "step" or "index" of the nilpotent Lie algebra $\gf$.
\end{definition}

By definitions above, finite dimensional graded Lie algebra is nilpotent. While the reverse is not always true.
\begin{remark}
    Not all nilpotent Lie algebras can be equipped with grading structure. On the other hand, the grading structure is not always unique. \cite[Remark 3.1.6.]{FischerRuzhansky2016}
\end{remark}

Furthermore, $\gf$ can be equipped with an action of $\mathbb{R}^{\times}$ which we denote by $\delta$ and call it dilation, with formula 
\begin{equation}\label{eq-dilation}
    \delta_{t}(Y_j) = t^jY_j,\quad Y_j \in V_j,\quad t>0.
\end{equation}

These are Lie algebra isomorphisms. Since $G$ is simply connected, by Lie theory, each $\delta_{t}$ will induce an unique Lie group isomorphism which we also denote by $\delta_{t}$. It satisfies  
\[
\delta_{t}\exp(Y_1+\cdots+Y_s) = \exp(t Y_1+t^2 Y_2+\cdots+t^s Y_s),
\]
where $Y_j \in V_j$ ({$Y_j=0$, if $V_j = \{0\}$}), and $\exp(Y_1+\ldots+Y_s)$ is a generic element of $G$, {due to \eqref{eq-stratfied-decomp}}. 
Since $G$ is nilpotent, exponential map is isomorphic between Lie algebra and Lie group, thus every element in $G$ can be expressed in the form $\exp(Y_1+\ldots+Y_s)$.

On homogeneous Lie group (Lie group equipped with dilations), \cite{HS90} guarantees there always exists a homogeneous norm. We fix this homogeneous norm for the rest of paper.
\begin{definition}\label{def-homo-metric}
	Denote the homogeneous norm on $G$ described above, by $|\cdot|_{\hom}$. It satisfies $|\delta_t g |_{\hom} = t| g|_{\hom}$ and $|g|_{\hom}=|g^{-1}|_{\hom}$, for all $g\in G, t\in \R^{\times}$. It induces a right-translation-invariant, homogeneous metric by $\dist(g_1, g_2) = |g_1g_2^{-1}|_{\hom}$ for any $g_1,g_2 \in G$.
\end{definition}

\begin{remark}
    We use the term "graded Lie algebra" to maintain consistency with the literature, such as, \cite[Definition 3.1.1.]{FischerRuzhansky2016}. One can also call this algebra  $\Z_s$ graded, $\mathbb{N}$ graded, or $\Z$ graded with additional conditions. However, we caution the reader not to confuse this with the "super graded Lie algebra" commonly encountered in differential geometry, where the prefix "super" is sometimes omitted. 
\end{remark}

\begin{definition}\label{def-generator-basis}
    For a graded Lie algebra $\gf,$ we say that a set $\{X_j\}_{j=1}^{n'}$,$X_i \in \gf$
    is a \emph{set of preferred generators} if
    \begin{enumerate}[{\rm (i)}]
        \item\label{enu-generator-basis-1} $\{X_j\}_{j=1}^{n'}$ are linearly independent,
        \item\label{enu-generator-basis-2} $\{X_j\}_{j=1}^{n'}$ generates $\gf$,
        \item\label{enu-generator-basis-3} Each $X_j$ is homogeneous, that is, $X_j \in V_{v_j}$, for some $v_j \geqslant 1$. 
    \end{enumerate}
    We call $v_j$ the degree of $X_j,$ denoted by $\deg(X_j) = v_j$. Set $v_{\gf} = \max_{j}{v_j}$.
\end{definition}
\begin{remark}
    By \cite[Lemma 2.2]{terElstRobinson97}, there exists a preferred generating set for any graded Lie algebra.
\end{remark}

}

\subsection{The universal enveloping algebra}\label{sec-universal-enve-alg}

Let us recall the definition and properties of the universal enveloping algebra of a Lie algebra. Let us fix a Lie algebra $\gf$ with corresponding Lie group $G$. Let $d = \mathrm{dim}(\gf) < \infty$.
\begin{definition}
    For Lie algebra $\gf$, the universal enveloping algebra is $\Uc(\gf) = T(\gf)/J$.
    Here, 
    \begin{equation*}
        T(\gf) := \oplus_{k=0}^{\infty} \gf^{\otimes k} = \C \oplus \gf \oplus (\gf \otimes \gf) \oplus \cdots
    \end{equation*}
    is the tensor algebra of $\gf$, $J\subseteq T(\gf)$ is the two-sided ideal generated by 
    \begin{equation}\label{eq-universal-envelop-alg-realtion}
        X \otimes Y - Y\otimes X - [X,Y], \quad X, Y\in \gf.
    \end{equation}
\end{definition}
It is well-known that $\Uc(\gf)$ can be identified with the space of right-invariant differential operators on $G.$ Let $X \in \gf$, its action $\lambda(X)$ on $u \in L_2(G)$ is defined as 
\begin{equation}\label{eq-def-right-action-Lie-alg}
    \lambda(X)u(g) = \frac{d}{dt}\Big|_{t=0} u(e^{-tX}g).
\end{equation}
Since $\lambda(X)$ is the generator of a unitary semigroup, Stone's theorem implies that $\lambda(X)$ is anti-self-adjoint. The representation $\lambda$ of $\gf$ on $L_2(G)$ extends to an algebra isomorphism from $\Uc(\gf)$ to the unital subalgebra of linear maps $C^\infty(G)\to C^\infty(G)$ generated by $\lambda(X_1),\ldots,\lambda(X_{n'})$, $X_1,...,X_{n'}$ are defined in Definition \ref{def-generator-basis}, for graded $\gf$.

Indeed, the Poincar\'e-Birkhoff-Witt theorem asserts that if $\{e_1,\ldots,e_d\}$ is a basis for $\gf,$ then the set of monomials $\{e_1^{\otimes k_1}\otimes \cdots\otimes e_{d}^{\otimes k_d}+J\}_{k_1,\ldots,k_d\geqslant 0}$ is a basis for $\Uc(\gf)$. More details can be found in  \cite[Proposition 1.9, Corollary 1.10, p.108]{Helgason-book-1978}.

Recall that we assume a grading on the Lie algebra $\gf$,\ $\gf = \bigoplus_{j=1}^s \gf_j,$ as in Definition \ref{def-stratfied-g}, and that we have fixed a choice of preferred generating set $\{X_j\}_{j=1}^{n'}.$ Recall Definition \ref{def-generator-basis} and \eqref{eq-def-diff-op}, given a word $\alpha=\alpha_1\cdots\alpha_k$ where $\alpha_j \in \{1,\ldots,n'\},$ denote
\begin{equation}\label{eq-def-X-alpha}
    X^{\alpha}:= X_{\alpha_1}X_{\alpha_2}\cdots X_{\alpha_k} \in  \Uc(\gf).
\end{equation}
If $\alpha$ is the empty word, put $X^{\alpha} = 1 \in \C \subset \Uc(\gf).$
Given a word $\alpha,$ recall that $\len(\alpha)$ denotes its weighted length \eqref{eq-def-len}.
\begin{remark}
    Notation \eqref{eq-def-X-alpha} is similar to that of Folland \cite[p.190]{Folland1975} and Helffer--Nourrigat \cite[Equation 1.6]{HelfferNourrigat1985}. Some other sources use a different spanning set for $\Uc(\gf),$ e.g. Fischer--Ruzhansky \cite[p.102]{FischerRuzhansky2016}.
\end{remark}

\begin{lemma}\label{lem-graded-basis-U}
    Let $\gf$ be a graded Lie algebra as above. 
    Then $\Uc(\gf) =  \mathrm{span}\{X^{\alpha}\},$ where $\alpha$ ranges over words in $\{1,\ldots,n'\}.$
\end{lemma}
\begin{proof}
    Extend $\{X_1,\ldots,X_{n'}\}$ to a basis of $\gf,$ say $\{X_1,\ldots,X_{d}\},$
    where $d = \mathrm{dim}(\gf).$ By the Poincar\'e-Birkhoff-Witt theorem, 
    $\Uc(\gf)$ is spanned by monomials of the form $X_1^{k_1}\cdots X_{d}^{k_d}$ for $k_1,\ldots,k_{d}\geqslant 0.$  
    
    By assumption, $\{X_j\}_{j=1}^{n'}$ generates the Lie algebra $\gf$. Thus, each $X_j$ for $j>n'$ is a linear combination of commutators of $\{X_j\}_{j=1}^{n'}$.
    In particular, it is a linear combination of products of $\{X_j\}_{j=1}^{n'}$.
    Hence, every $X_1^{k_1}\cdots X_{d}^{k_d}$ belongs to $\mathrm{span}\{X^{\alpha}\}$, where $\alpha$ ranges over words in $\{1,\ldots,n'\}.$
\end{proof}

Notice that the set $\{X^{\alpha}\}$ is not necessarily linearly independent, thus Lemma \ref{lem-graded-basis-U} is much weaker than the Poincar\'e-Birkhoff-Witt theorem.

\begin{definition}
    Let $\gf$ be a graded Lie algebra as above.
    For $m\geqslant 0,$ we denote $\Uc_m(\gf)$ for the subspace spanned by $\{X^{\alpha}\}$, $\alpha$ ranges over words in $\{1,\ldots,n'\}$, and $\len(\alpha)$ at most $m$.
\end{definition}
The definition of $\Uc_{m}(\gf)$ is independent of the choice of preferred generator set $\{X_j\}_{j=1}^{n'}$ . Indeed, if $\{Y_1,\ldots,Y_{n''}\}$ is another preferred generator set, as in Definition \ref{def-generator-basis}, then each $Y_j$ is a linear combination of commutators of $\{X_k\}_{k=1}^{n'}$. 
Combined with \eqref{eq-stratified-relation}, we have 
\[ Y_j \in \mathrm{span}\{X^{\alpha}\}_{\mathrm{len}(\alpha)=\deg(Y_j)}.\]
It follows that $ Y^{\beta} \in \mathrm{span}\{X^{\alpha}\}_{\mathrm{len}(\alpha)=\mathrm{len}(\beta)}.$

\section{Sobolev spaces on a graded Lie group}\label{sec-3}

In this section, first, we discuss some preliminary material concerning function spaces and differential operators on {graded} Lie groups. Our main sources here are \cite{FischerRuzhansky2016,Folland1975,FollandStein1982,RothschildStein1976}.

Then, by taking advantage of the Lie group structure and the translation invariant metric, 
we build a special partition of unity, such that each function share the same format and their support are translated from each other. See Lemma \ref{quasi_metric_space_construction} and Construction \ref{partition_of_unity}. It will play an essential role in inheriting results from constant coefficient differential operator to the ones with varying coefficients.

The main result in this section is Theorem \ref{sobolev localization theorem} which says the Folland-Stein Sobolev norm whose definition is recalled below can be localized using the partition of unity.

This section is organized as follows:

In \ref{subsec-3.1}, we recall the Folland-Stein Sobolev spaces and discuss their interpolation properties.

In \ref{subsec-3.2}, we construct a partition of unity for $G.$

In \ref{subsec-3.3}, we prove Theorem \ref{sobolev localization theorem}.

\subsection{Sobolev Spaces}\label{subsec-3.1}

In this Section, we define Sobolev spaces on graded Lie groups, and exhibits some essential properties. For further details, see \cite[Chapter 4]{FischerRuzhansky2016} or \cite{Folland1975} for the special case when $\gf$ is stratified.

As in Definition \ref{def-generator-basis}, a set of preferred generators $\{X_j\}_{j=1}^{n'}$ is fixed throughout the paper, $\{v_j\}_{j=1}^{n'}$ are their degrees, and $v = \lcm(\{v_j\}_{j=1}^{n'})$. Recall \eqref{eq-def-new-Laplacian}, 
    \begin{equation*}
        \Delta_{G} := - \sum_{j=1}^{n'} (-1)^{\frac{v}{v_j}}X_j^{\frac{2v}{v_j}}.
    \end{equation*}
    This operator is Rockland, that is, for every non-trivial unitary irreducible representation $\pi$ of $G$, $\pi(\Delta)$ is injective on smooth vectors \cite[Definition 4.1.1. and Corollary 4.1.10.]{FischerRuzhansky2016}. The operator $-\Delta_G$ is positive definite and self-adjoint on $L_2(G)$ \cite[Proposition 4.1.15.]{FischerRuzhansky2016}.
    Knowing the self-adjointness of $\Delta_G,$ we define a scale of Sobolev spaces $\{W^s_2(G)\}_{s\in \mathbb{R}}$ with the norms
    \begin{equation}\label{eq-sobolev-norm}
        \|u\|_{W^s_2(G)} := \|(1-\Delta_G)^{\frac{s}{2v}}u\|_{L_2(G)}, \quad u \in C_c^{\infty}(G).
    \end{equation}
    Here, the power $(1-\Delta_G)^{\frac{s}{2v}}$ is understood in the sense of spectral theory.
    The Sobolev space is defined as the closure of $C^\infty_c(G)$ in $\Sc'(G)$ with the norm $\|\cdot\|_{W^s_2(G)}.$

We also consider the homogeneous Sobolev semi-norm  $\Vert \cdot  \Vert_{\mathring{W}^s_2(G)}$, $s\in \mathbb{R}$ defined by
\begin{equation}\label{eq-sobolev-norm-homogeneous}
    \|u\|_{\mathring{W}^s_2(G)} := \|(-\Delta_G)^{\frac{s}{2v}}u\|_{L_2(G)}, \quad u \in \mathrm{dom}((-\Delta_G)^{\frac{s}{2v}}).
\end{equation}
We will mostly abbreviate $\Delta_G$ as $\Delta.$

We list a few useful properties of these Sobolev spaces:
\begin{enumerate}[{\rm (i)}]
    \item Both inhomogeneous and homogeneous Sobolev norms $W^s_2(G)$ and $\mathring{W}^s_2(G)$ are independent of the choice of preferred generating set $\{X_j\}_{j=1}^{n'}$ \cite[Theorem 4.4.20]{FischerRuzhansky2016}.
    \item The Schwartz space $\Sc(G)$ is dense in $W^s_2(G)$ for all $s\in\mathbb{R}$ \cite[Lemma 4.4.1]{FischerRuzhansky2016}.
    \item \label{item-Sobolev-emb}
    	We have $W^{s_1}_2(G) \subseteq W^{s_2}_2(G)$ for $s_1 \geqslant s_2 \in \R$ \cite[Theorem 4.4.3]{FischerRuzhansky2016}.
    \item For all $s_0,s_1\in \mathbb{R}$ and $0 < \theta < 1$ we have (up to equivalence of norms)
        \begin{equation}\label{interpolation}
            (W^{s_0}_2(G),W^{s_1}_2(G))_{\theta} = W^{s_\theta}_2(G),\quad s_\theta=(1-\theta)s_0+\theta s_1,
        \end{equation}
        where $(\cdot,\cdot)_{\theta}$ is the functor of complex interpolation. 
        In particular,
        \begin{equation}\label{eq-interpolation-norm}
            \|u\|_{W^{s_\theta}_2(G)} \leqslant \|u\|_{W^{s_0}_2(G)}^{1-\theta}\|u\|_{W^{s_1}_2(G)}^{\theta},\quad u \in W^{\max\{s_0,s_1\}}_2(G).
        \end{equation}
        See \cite[Theorem 4.4.28]{FischerRuzhansky2016}.
    \item Recall that $v$ is the least common multiple of $\{v_j\}_{j=1}^{n'}$, the degrees of preferred generators. For $s \in 2v\cdot \Z_{+}$, the Sobolev norm $\Vert \cdot \Vert_{W^{s}_2(G)}$ is equivalent to the following norm:
        \begin{equation}\label{integer sobolev norm}
            u \mapsto \big(\sum_{\len(\alpha)\leqslant s}\|X^{\alpha}u\|_{L_2(G)}^2\big)^{\frac12},
            \quad u \in W^s_2(G).
        \end{equation}
        For $s \in 2v\cdot\Z^+$, we have
        \begin{equation}\label{eq-Sobolev-alt-def}
            W^s_2(G) = \left\{u\in L_2(G) \;:\; \ X^{\alpha}u\in L_2(G), \ \len(\alpha)\leqslant s \right\}.
        \end{equation}
        The equivalence of \eqref{eq-sobolev-norm} and \eqref{integer sobolev norm} was originally proved by Helffer and Nourrigat  \cite[Estimate (6.1)]{HelfferNourrigat1979}, and restated in \cite[Corollary 4.1.14.]{FischerRuzhansky2016}.
    \item For all $s \in 2v \cdot \mathbb{Z}_+,$ $W^{-s}_2(G)$ coincides with the Banach dual of $W^s_2(G).$ Concretely, $W^{-s}_2(G)$ is identified with the space of distributions $u\in \Sc'(G)$ such that there is a constant $C$ such that for all $\phi \in \Sc(G)$ we have
    \[
        |(u,\phi)|\leqslant C\|\phi\|_{W^s_2(G)}.
    \]
    The least constant $C$ is a norm equivalent to $\|\cdot\|_{W^{-s}_2(G)}.$ 
    For a stratified Lie group $G$, this follows from \cite[Theorem 3.15(v)]{Folland1975}, in particular see the Remark below the proof of \cite[Proposition 4.1]{Folland1975}.
\end{enumerate}

It is immediate from the above definitions that the order of a differential operator in $\Uc(\gf)$ coincides with its order as a mapping between Sobolev spaces.
\begin{lemma}\label{lemma-X_alpha-bdd}
For every $s \in \mathbb{R},$ an element $D \in \Uc_m(\gf)$ extends by continuity to a bounded linear map
\begin{equation*}
D:W^s_2(G)\to W^{s-m}_2(G).
\end{equation*}
\end{lemma}

\subsection{Partition of Unity}\label{subsec-3.2}

\begin{definition}\label{C_infty_definition}
Let $C_b(G)$ denote the space of bounded continuous functions on $G.$
For $k\geqslant 0,$ let $C^k_b(G)$ denote the space of $f \in C_b(G)$ such that for all words $\alpha$ with $\len(\alpha)\leqslant k$ we have
\[
X^{\alpha}f \in C_b(G).
\]
Define $\|f\|_{k,b} = \sup_{\len(\alpha) \leqslant k} \|X^{\alpha}f\|_{L_{\infty}(G)}$ and let $C^\infty_b(G) = \bigcap_{k\geqslant 0} C^k_b(G)$. 
\end{definition}

\begin{lemma}\label{multiplication_lemma}
If $\phi \in C^\infty_b(G),$ then the multiplier operator $M_{\phi}$ is bounded from $W^s_2(G)$ to $W^s_2(G)$ for all $s\in\mathbb{R}$ with norm no greater than $C_s \|\phi\|_{\lceil|s|\rceil,b}.$
\end{lemma}
\begin{proof}
We prove this initially { for $s\in 2v\Z_+.$} By the Leibniz rule if $\len(\alpha)\leqslant s,$ there are constants $c_{\alpha,\beta}$ such that
\begin{equation}\label{eq-X-alpha-on-varphi-function}
X^{\alpha}(\phi u)=\sum_{\alpha = \beta\beta'}c_{\alpha,\beta}\cdot (X^{\beta}\phi)\cdot X^{\beta'}u
\end{equation}
where the sum is over all words $\beta$ and $\beta'$ such that $\alpha$ is the concatenation $\beta\beta'.$ In particular, $\len(\beta)+\len(\beta')=\len(\alpha).$
By the triangle inequality,
\begin{align*}
\|X^{\alpha}(\phi u)\|_2&\leqslant \sum_{\alpha=\beta\beta'}c_{\alpha,\beta}\max_{\len(\beta)\leqslant s}\|X^{\beta}\phi\|_{\infty}\max_{\len(\beta')\leqslant s}\|X^{\beta'}u\|_2\\
&\leqslant \sum_{\alpha=\beta\beta'}c_{\alpha,\beta}\|\phi\|_{s,b}\|u\|_{W^{\len(\alpha)}_2} \leqslant c_{\alpha}\|\phi\|_{s,b}\|u\|_{W^{s}_2}.
\end{align*}
Therefore,
\begin{align*}
\|M_{\phi}u\|_{W^s_2}^2 &= \sum_{\len(\alpha)\leqslant s}\|X^{\alpha}(\phi u)\|_2^2\leqslant\sum_{\len(\alpha)\leqslant s}c_{\alpha}^2\|\phi\|_{s,b}^2\|u\|_{W^{s}_2}^2\\
&\leqslant \big( \sum_{\len(\alpha)\leqslant s}c_{\alpha}^2 \big) \cdot \|\phi\|_{s,b}^2\|u\|_{W^s_2}^2.
\end{align*}
This yields the assertion for integer $s\in 2v\Z_+.$ The case for general $s\geqslant 0$ follows from 
{interpolation \eqref{interpolation}} between the $2v\lfloor s/(2v)\rfloor$ and $2v\lceil s/(2v)\rceil$ cases.

For $s<0,$ the assertion follows by an easy duality argument.
\end{proof}

The following is related to \cite[Lemma 5.7.5]{FischerRuzhansky2016}, and also \cite[Lemma 6.2]{MSZ-stratified-23}, but a proof is supplied for convenience.
\begin{lemma}\label{quasi_metric_space_construction}
{Let $(X,d)$ be a {separable} metric space, thus $X$ possesses a Borel measure $\mu$. If there exists a constant $\delta>0$ such that}
$$\mu(B(x,r))=r^{\delta},\quad x\in X,\quad r > 0.$$
{Then,} for every $\epsilon>0,$ there exists a set $\{x_i\}_{i\in I}\subset X$ such that
\begin{enumerate}[{\rm (i)}]
\item{} $\{B(x_i,\epsilon)\}_{i\in I}$ covers $X,$ and 
\item{} A fixed ball $B(x_i,\epsilon)$  {intersects} at most $5^{\delta}$ of balls $\{B(x_i,\epsilon)\}_{i\in I}.$
\item{} Moreover, for every $N\in\mathbb{N}$, a fixed ball $B(x_i, N\epsilon)$ {intersects} at most $(4N+1)^{\delta}$ of balls $\{B(x_i,N\epsilon)\}_{i\in I}.$
\end{enumerate}
\end{lemma}
\begin{proof} Fix $\epsilon>0$ and let $\{B(x_i,\frac{\epsilon}{2})\}_{i\in I}$ be a maximal disjoint collection of balls in $X.$ 

We claim that $\{B(x_i,\epsilon)\}_{i\in I}$ covers $X.$ Indeed, let $x\in X.$ By maximality, $B(x,\frac{\epsilon}{2})\cap B(x_i,\frac{\epsilon}{2})\neq\emptyset$ for some $i\in I.$ By triangle inequality, $d(x,x_i)<\epsilon$ or, equivalently, $x \in B(x_i,\epsilon).$ This proves that $\{B(x_i,\epsilon)\}_{i\in I}$ covers $X.$ 

(ii) is obviously a special case of (iii), thus we only prove (iii) here. Fix a ball $B(x_i, N\epsilon)$,  suppose that there exist $\{x_{i_k}\}_{k=1}^M$ such that
\begin{equation*}
B(x_i, N\epsilon) \bigcap B(x_{i_k}, N\epsilon) \neq \emptyset. 
\end{equation*}
This is equivalent as $d(x,x_{i_k}) < 2N\epsilon$ for all $1\leqslant k\leqslant M.$ Therefore
$$\bigcup_{k=1}^M B(x_{i_k},\frac{\epsilon}{2}) \subset B(x_i,\frac{(4N+1)\epsilon}{2}).$$
Hence,
$$\mu\Big(\bigcup_{k=1}^M B(x_{i_k},\frac{\epsilon}{2})\Big)\leqslant \mu( B(x_i,\frac{(4N+1)\epsilon}{2}))=(\frac{(4N+1)\epsilon}2)^{\delta}.$$   
Since the sets $\{B(x_i,\frac{\epsilon}{2})\}_{i\in I}$ are pairwise disjoint, it follows that
$$\mu\Big(\bigcup_{k=1}^M B(x_{i_k},\frac{\epsilon}{2})\Big) = \sum_{k=1}^M \mu(B(x_{i_k},\frac{\epsilon}{2}))=M\cdot (\frac{\epsilon}{2})^{\delta}$$
and therefore
$$M\cdot (\frac{\epsilon}{2})^{\delta} \leqslant (\frac{(4N+1)\epsilon}2)^{\delta}.$$
In other words, $M \leqslant (4N+1)^{\delta}.$ Hence, $B(x_i, N \epsilon)$ intersects at most $(4N+1)^{\delta}$ elements of $\{B(x_i,N\epsilon)\}_{i\in I}.$
\end{proof}

We will frequently refer to a partition of unity for $G$ having the following properties.
\begin{construction}\label{partition_of_unity} Recall Definition \ref{def-homo-metric} that $\dist$ is the fixed translation-invariant homogeneous metric on $G,$ and let $\mu$ be the Haar measure. By dilation and translation invariance of the Haar measure, we have $\mu(B(g,r))=r^{d_{{\rm hom}}}$ for every $g\in G$ and for every $r>0,$ up to an irrelevant normalisation. Hence, the assumptions of Lemma \ref{quasi_metric_space_construction} are satisfied. Notice that the Haar measure is just the lift of Lebesgue measure on $\gf$ via $\exp$ map.

Let $\psi\in C^{\infty}_c(G)$ equal $1$ near $0.$ Fix $\epsilon>0$ and let $\{g_n\}_{n\in I}$ be a set given by Lemma \ref{quasi_metric_space_construction}. As $G$ is separable, the latter set is countable and is denoted by $\{g_n\}_{n\geqslant 0}$. Define $\Psi=\{\psi_n\}_{n\geqslant0}$ by setting
by setting
\begin{equation}\label{eq-partition-of-unity}
\psi_n(g)=\frac{\psi(gg_n^{-1})}{(\sum_{k\geqslant0}\psi^2(gg_k^{-1}))^{\frac12}},\quad n\geqslant 0.
\end{equation}
We do right shift here to be compatible with right-invariant operators.
Denote the sum in the denominator of $\psi_n$ by $$\theta(g)=\sum_{k\geqslant0}\psi^2(g g_k^{-1}),\quad g\in G.$$
\end{construction}
Clearly, $\theta$ is finite and by construction we have
\[
\sum_{n=0}^\infty \psi_n(g)^2 = 1,\quad g \in G.
\]

\begin{lemma}\label{smooth_partition_of_unity} 
As defined in Construction \ref{partition_of_unity}, $\theta \in C_b^{\infty}(G)$. For every word $\alpha$ the function $\sum_{n=0}^{\infty}|X^{\alpha}\psi_n|^2$ is bounded.
\end{lemma}
\begin{proof} 
Suppose that $\psi$ equals $1$ in $B(0,\epsilon)$ and that $\psi$ is supported in $B(0,N\epsilon).$
Note that, in the neighbourhood of any given $g\in G,$ at most $(4N+1)^{d_{{\rm hom}}}$ summands are non-zero. Thus, $\theta\in C^{\infty}_b(G)$ and $1\leqslant\theta\leqslant (4N+1)^{d_{{\rm hom}}}.$ Hence, $\theta^{-\frac12}\in C^{\infty}_b(G).$ 

The assertion follows now from Lemma \ref{quasi_metric_space_construction} and the Leibniz rule.
\end{proof}

\begin{corollary}\label{cor-operator-estimate-partition-unity}
As defined in Construction \ref{partition_of_unity}, for the partition of unity $\{\psi_n\}$, we have $\psi_n \in C_b^{\infty}(G)$ for each $n\in \mathbb{N}$. For all $k\geqslant 0$,
\begin{equation*}
\sup_n \Vert \psi_n \Vert_{k, b} < \infty.
\end{equation*}
Furthermore, for all $s\in\R$, for every word $\alpha$, the multiplier operators $M_{X^{\alpha}\psi_n}$ have
\begin{equation*}
\sup_n \Vert M_{X^{\alpha}\psi_n} \Vert_{W^s_2(G) \to W^s_2(G)} < \infty.
\end{equation*}
Also for $\theta$ in Construction \ref{partition_of_unity}, $\Vert M_{\theta} \Vert_{W^s_2(G) \to W^s_2(G)} < \infty$.
\end{corollary}

\begin{proof}
Combine Lemma \ref{smooth_partition_of_unity} with Definition \ref{C_infty_definition} and Lemma \ref{multiplication_lemma}.
\end{proof}


\subsection{Localisation of Sobolev Norms}\label{subsec-3.3}

The following theorem shows how the Folland-Sobolev norms can be localised by translations of smooth functions. The analogous statement for classical function spaces is well-known, see e.g. \cite[Section 2.4.7]{Triebel-2}.
\begin{theorem}\label{sobolev localization theorem}
Let $\Psi=\{\psi_n\}_{n=0}^\infty$ be a partition of unity as in Construction \ref{partition_of_unity}. For every $s\in\mathbb{R}$, there exists a positive constant $c_{\Psi,s}$, such that 
\begin{equation}\label{eq-local-global-norm-equiv}
c_{\Psi,s}^{-1}\|u\|_{W^s_2(G)}\leqslant\left(\sum_{n=0}^\infty \|\psi_n u\|_{W^s_2(G)}^2\right)^{\frac12}\leqslant c_{\Psi,s}\|u\|_{W^s_2(G)},\quad u\in W^s_2(G).
\end{equation}
\end{theorem}
Theorem \ref{sobolev localization theorem} will be proved by relating the inequality to boundedness of the linear operators $B_{\alpha}$ (Lemma \ref{first localization lemma}) , $A_{{m}}$ (Lemma \ref{a2m lemma}) , and $A_z$ (Lemma \ref{az lemma}) .

Let $\Psi=\{\psi_n\}_{n=0}^\infty$ be a partition of unity as in Construction \ref{partition_of_unity}. 
Let $\alpha$ be a word in $\{1,\ldots, d_1\}$. 
We define a map $B_{\alpha}:L_2(G)\to \bigoplus_{n\geqslant0}L_2(G)$, here the direct sum is the Hilbert space direct sum, meaning the completion of algebraic direct sum under Hilbert space norm. 
Let $B_{\alpha}$ {be the direct sum}:
$$B_{\alpha}=\bigoplus_{n\geqslant0}M_{X^{\alpha}\psi_n}.$$
In particular, when $\len(\alpha) = 0$, we denote
\begin{equation*}
B_{0}=\bigoplus_{n\geqslant0}M_{\psi_n}.
\end{equation*}

\begin{lemma}\label{first localization lemma} 
For all words $\alpha,$ the map $B_{\alpha}$ defined above is bounded.
\end{lemma}
\begin{proof} 
Since $|u|^2$ is integrable, the dominated convergence theorem implies
\begin{align*}
\sum_{n=0}^\infty \|(X^{\alpha}\psi_n) u\|_2^2
& = \sum_{n=0}^\infty \int_G |X^{\alpha} \psi_n|^2 |u|^2 d\mu
= \int_G \sum_{n=0}^\infty  |X^{\alpha} \psi_n|^2 |u|^2 d\mu \\
& = \|(\sum_{n=0}^{\infty}|X^{\alpha}\psi_n|^2)^{\frac12}u\|_2^2\leqslant\|\sum_{n=0}^{\infty}|X^{\alpha}\psi_n|^2\|_{\infty}\|u\|_2^2.
\end{align*}

The assertion follows now from Lemma \ref{smooth_partition_of_unity}.
\end{proof}

Notice that $B_{\alpha}$ is injective when $\len(\alpha) = 0$ while $B_{\alpha}$ is not necessarily injective when $\len(\alpha) > 0$.

Now, we go one step further, { and construct an operator $A_{m}$ from operator $B_{\alpha}$ as follows.}

\begin{lemma}\label{a2m lemma} 
Let $\{\psi_n\}_{n=0}^\infty$ be the partition of unity in Construction \ref{partition_of_unity}.
Let $m\in\mathbb{Z}$. Consider the mapping $A_{m}:\Sc(G)\to \bigoplus_{n\geqslant0}L_2(G)$
defined by the formula
\begin{equation}\label{eq-def-a2m}
    A_{m}=\bigoplus_{n\geqslant0}(1-\Delta)^mM_{\psi_n}(1-\Delta)^{-m}.
\end{equation}
{Then $A_m$ has a bounded extension to $L_2(G),$ which we still denote $A_m.$}
\end{lemma}
\begin{proof} Suppose first $m\geqslant 0.$ By \eqref{eq-X-alpha-on-varphi-function}, we can always move multiplication operator to the left and write
$$(1-\Delta)^mM_{\psi_n}=\sum_{\len(\alpha)\leqslant {2v \cdot m}}c_{m,\alpha}^{(1)}M_{X^{\alpha}\psi_n}P_{m,\alpha},$$
where $P_{m,\alpha}$ is a constant coefficient differential operator of order {$2v \cdot m$} or less. We write
\begin{align*}
A_{m}
& = \bigoplus_{n\geqslant 0}\Big[ \sum_{\len(\alpha)\leqslant {2v \cdot m}}c_{m,\alpha}^{(1)}M_{X^{\alpha}\psi_n} P_{m,\alpha}\Big](1-\Delta)^{-m}\\
& =	\sum_{\len(\alpha)\leqslant {2v \cdot m} }c_{m,\alpha}^{(1)}B_{\alpha} (P_{m,\alpha}(1-\Delta)^{-m}),
\end{align*}
where $B_{\alpha}$ is the operator in Lemma \ref{first localization lemma}. The boundedness of $A_{m}$ follows now from Lemma \ref{first localization lemma}.

Suppose now that $m\leqslant 0.$ Using similar argument as above, we obtain
$$M_{\psi_n}(1-\Delta)^{-m}=\sum_{\len(\alpha)\leqslant {2v \cdot m}}c_{m,\alpha}^{(2)}Q_{m,\alpha}M_{X^{\alpha}\psi_n},$$
where $Q_{m,\alpha}$ is a differential operator of order {$-2v \cdot m$} or less. We write
$$A_{m}=\sum_{\len(\alpha)\leqslant {2v \cdot m}}c_{m,\alpha}^{(2)}\Big(\bigoplus_{n\geqslant 0}(1-\Delta)^mQ_{m,\alpha}\Big)\circ B_{\alpha},$$
where $B_{\alpha}$ is the operator in Lemma \ref{first localization lemma}, thus bounded.
We also notice that each $\bigoplus_{n\geqslant 0}(1-\Delta)^mQ_{m,\alpha}$ is a bounded mapping, so $A_{m}$ is bounded.
\end{proof}

Now we extend the parameter in definition of $A_{m}$ from integer $m$ to complex number $z$, and prove the boundedness is still valid.  

\begin{lemma}\label{az lemma} Let $z\in\mathbb{C},$ and let $\{\psi_n\}_{n=0}^\infty$ be a partition of unity in Construction \ref{partition_of_unity}. Consider the mapping $A_z:\Sc(G)\to \bigoplus_{n\geqslant0}L_2(G)$
defined by the formula
$$A_z=\bigoplus_{n\geqslant0}(1-\Delta)^{{z}}M_{\psi_n}(1-\Delta)^{{-z}}$$
{Then $A_z$ has a bounded extension to $L_2(G),$ which is also denoted by $A_z$.}
\end{lemma}
\begin{proof} Let $z={m} +it,$ $m\in\mathbb{Z}.$ We have
$$A_z=(\bigoplus_{n\geqslant0}(1-\Delta)^{{it}})\circ A_{m}\circ (1-\Delta)^{{-it}}.$$

{
Since $-\Delta$ is positive and self-adjoint, the operators $(1-\Delta)^{it}$ and $\bigoplus_{n\geqslant 0} (1-\Delta)^{it}$ are unitary on $L_2(G)$ and $\bigoplus_{n\geqslant 0} L_2(G)$ respectively. 
}
Thus,
$$\|A_z\|_{L_2(G)\to \bigoplus_{n\geqslant0}L_2(G)}=\|A_{2m}\|_{L_2(G)\to \bigoplus_{n\geqslant0}L_2(G)}.$$
The assertion follows from Lemma \ref{a2m lemma} and the Hadamard 3 lines theorem.
\end{proof}

\begin{lemma}\label{az adjoint lemma} Let $z\in\mathbb{C}$ and let $A_z$ be as in Lemma \ref{az lemma}.
Its adjoint is given by the formula
$$A_z^{\ast}:(v_n)_{n\geqslant0}\to\sum_{n\geqslant0}(1-\Delta)^{{-\bar{z}}}M_{\psi_n}(1-\Delta)^{{\bar{z}}}v_n,\quad (v_n)_{n\geqslant0}\in \bigoplus_{n\geqslant0}L_2(G).$$
\end{lemma}
\begin{proof} Let $u\in L_2(G)$ and let $v = (v_n)_{n\geqslant 0} \in \bigoplus_{n\geqslant 0} L_2(G).$ We have
\begin{align*}
\langle A_zu,v\rangle 
&=\sum_{n\geqslant0}\langle (1-\Delta)^{{-\bar{z}}}M_{\psi_n}(1-\Delta)^{{-\bar{z}}}u,v_n\rangle\\
&=\sum_{n\geqslant0}\langle u,(1-\Delta)^{{-\bar{z}}}M_{\psi_n}(1-\Delta)^{{\bar{z}}}v_n\rangle\\
&=\langle u,\sum_{n\geqslant0}(1-\Delta)^{{-\bar{z}}}M_{\psi_n}(1-\Delta)^{{\bar{z}}}v_n\rangle.
\end{align*}
\end{proof}

\begin{proof}[Proof of Theorem \ref{sobolev localization theorem}] The boundedness of $A_{{\frac{m}{2v}}}:L_2(G)\to \bigoplus_{n\geqslant0}L_2(G)$ as defined in Lemma \ref{az lemma} is equivalent to the boundedness of $B_0:W^m_2(G)\to \bigoplus_{n\geqslant 0}W^m_2(G),$ where $B_0$ is defined in Lemma \ref{first localization lemma}. This yields the right hand side inequality of \eqref{eq-local-global-norm-equiv} .

The boundedness of $A_{{-\frac{m}{2v}}}^{\ast}:\bigoplus_{n\geqslant0}L_2(G)\to L_2(G)$ is equivalent to the boundedness of operator 
\begin{align*}
B_0^*\colon  \bigoplus_{n\geqslant 0}W^m_2(G) & \to W^m_2(G),\\
(v_n)_{n\geqslant 0} & \mapsto \sum_{n\geqslant 0}\psi_n v_n.
\end{align*} 
Indeed, $A_{{-\frac{m}{2v}}}^{\ast}$ is bounded according to Lemma \ref{az adjoint lemma}. So is $B_0^*$.

Now, for any $u\in W^m_2(G)$, by boundedness of $B_0:W^m_2(G)\to \bigoplus_{n\geqslant 0}W^m_2(G)$ explained earlier, we have $(\psi_n u)_{n\geqslant0} \in \bigoplus_{n\geqslant 0}W_2^m(G)$. 
Since $\{\psi_n\}_{n=0}^\infty$ is the partition of unity as in Construction \ref{partition_of_unity}, we have $\sum_{n\geqslant0}\psi_n^2 = 1$. 
It follows that $B_0^* (\psi_n u)_{n\geqslant0} = \sum_{n\geqslant 0} \psi_n \cdot \psi_n u = u$. This  yields the left hand side  inequality of \eqref{eq-local-global-norm-equiv}.
\end{proof}


\section{Differential operators and elliptic regularity}\label{sec-4}

In this section, we prove our main result (Theorem \ref{elliptic regularity theorem}).  
{As a key step}, we reveal that, for differential operators as in Definition \ref{def-differential-op}, "ellipticity" ( uniformly Rockland ) in the sense of Definition \ref{definition of ellipticity_1} implies global elliptic estimates (Theorem \ref{forward elliptic estimate} and \ref{backward elliptic estimate}).

We prove these results first for constant coefficient differential operators, then varying the coefficients in a small neighbourhood, and then globalising using the partition of unity discussed in Section \ref{sec-3}.

This section is organised as follows:

In section \ref{sub-sec-diff-op}, we list a few definitions and define   uniformly Rockland   algebraically.

In section \ref{sub-sec-forward-elliptic} and \ref{sub-sec-backward-elliptic}, we show two different kinds of elliptic estimates, which we term "forward" and "backward" estimates. 

In section \ref{sub-sec-self-adjoint}, we prove all of main theorems of this paper.

In section \ref{example 4 section}, we discuss an example of  uniformly Rockland   operators.

\subsection{Differential Operator}\label{sub-sec-diff-op}

\begin{definition}\label{def-differential-op}
A differential operator of order at most $m$ on $G$ is a linear operator on $\Sc(G)$ given by
$$P=\sum_{\len(\alpha)\leqslant m} M_{a_{\alpha}}X^{\alpha}:\Sc(G)\to \Sc(G),$$
where every $a_{\alpha}$ belongs to $C^\infty_b(G)$, and $X^{\alpha}$ is defined in \eqref{eq-def-X-alpha}.

We denote by $P^\dagger$ the formal adjoint of $P,$ given by 
$$P^\dagger = \sum_{\len(\alpha)\leqslant m} (X^{\alpha})^\dagger M_{\overline{a_{\alpha}}} ,$$ 
where $(X^{\alpha})^\dagger = (-1)^kX_{\alpha_k}\cdots X_{\alpha_1}$, for $\alpha=\alpha_1\cdots\alpha_k$.  {Recall Definition \ref{C_infty_definition}, $a_{\alpha}$ and all its derivatives are bounded, which guarantees that $P^{\dagger}$ be continuous on the Schwartz space $\Sc(G)$.}

Let $\widetilde{P}:\Sc'(G)\to\Sc'(G)$ denote the extension of $P$ to distributions, defined on $\omega \in \Sc'(G)$ by
    \begin{equation}\label{def-P-dagger}
        \langle \widetilde{P}\omega,\phi\rangle = \langle \omega,P^\dagger\phi\rangle,\quad \phi\in \Sc(G).
    \end{equation}
\end{definition}

\begin{remark}\label{rmk-P-P-dagg-P-tilde-bdd}
    {Since all coefficients $a_{\alpha} \in C_{b}^{\infty}(G)$, the differential operator $P$ is continuous on $\Sc(G).$ Since $P^{\dagger}$ is of the same form, $\widetilde{P}$ is continuous on $\Sc'(G)$ with its usual topology.}
\end{remark}

We use the symbol $\dagger$ for the formal adjoint to avoid confusion with the adjoint $P^{\ast}$ in  Hilbert space $L_2(G)$,  which will have a slightly different meaning due to having a different domain.

By \eqref{def-P-dagger}, we notice that  $P$ coincides with $\widetilde{P}$ on $\Sc(G),$ $P^\dagger$ coincides with $P^{\ast}$ on $\Sc(G).$

\begin{lemma}\label{differential_operators_are_bounded} 
The distributional extension $\widetilde{P}$ of a differential operator $P$ of order at most $m$ restricts to a bounded linear operator from $W^{m+s}_2(G)$ to $W^s_2(G)$ for every $s\in \mathbb{R}.$
\end{lemma}
\begin{proof} Obvious.
\end{proof}

Now we introduce a way to "freeze" the coefficients of $P$ at a point $g\in G.$
\begin{definition}\label{def-P_g}
Let $P$ be a differential operator as in \eqref{eq-P-expression}. For any $g\in G$, we define the right-invariant constant coefficient differential operator $P_g \in \Uc(\gf)$ {as follows:}
\begin{equation*}
P_g=\sum_{\len(\alpha)\leqslant  m}a_{\alpha}(g)X^{\alpha}.
\end{equation*}
\end{definition}
It is not obvious that $P_g$ is well-defined. Indeed, the coefficients $X^{\alpha}$ are linearly dependent and there is no guarantee that the individual terms $a_{\alpha}(g)X^{\alpha}$ are uniquely determined by the operator $P.$ Nevertheless, their sum defines a unique element $P_g$ of $\Uc(\gf).$

To see this, it is enough to assume that $P = \sum_{\len(\alpha)\leqslant m} M_{a_{\alpha}}X^{\alpha} = 0.$ Fix a particular representation of $P$ as in \eqref{eq-P-expression} and consider $P_g$ coming from this representation. We have
\begin{equation*}
(P_gf)(g)=\sum_{\len(\alpha)\leqslant m}a_{\alpha}(g)\cdot (X^{\alpha}f)(g)=\Big(\sum_{\len(\alpha)\leqslant m} M_{a_{\alpha}}X^{\alpha}f\Big)(g)=(Pf)(g)=0
\end{equation*}
for every $f\in \Sc(G).$ Since the operator $P_g$ is right-invariant, it follows that $(P_gf)(g')=0$ for every $f\in\Sc(G)$ and for every $g'\in G.$ In other words, $P_g=0$ as required.

The principal part of a differential operator is its homogeneous component of highest degree. 
We define it analytically, and prove it is coherent with the classical definition (as in \cite[Equation, p.7]{HelfferNourrigat1985}).

\begin{definition}\label{def-P-top}
For differential operator $P$ of order at most $m$ on $G$, we define constant coefficient homogeneous differential operator $P^{\rm top}_g$ as follows:
\begin{equation*}
P_g^{{\rm top}}=\sum_{\len(\alpha)=m}a_{\alpha}(g)X^{\alpha}.
\end{equation*}
\end{definition}

Clearly, $P_g^{{\rm top}}$ is in the span of dilations of the operator $P_g.$ Since the operator $P_g$ is well defined, it follows that $P_g^{{\rm top}}$ is also well defined.

The following lemma shows two new operations $(\cdot)^{\dagger}$ and $(\cdot)^{\rm top}_g$ commute.

\begin{lemma}\label{top dag commute}
$(P_g^{{\rm top}})^\dag = (P^{\dag})_g^{\rm top}$.
\end{lemma}

\begin{proof}
Since taking the commutator of a left-invariant differential operator with a multiplication operator lowers the degree of differential operator, we have
\begin{align*}
P^\dagger 
& = \sum_{\len(\alpha)\leqslant m} (X^{\alpha})^\dagger \cdot M_{\overline{a_{\alpha}}}\\
& = \sum_{\len(\alpha)\leqslant m} M_{\overline{a_{\alpha}}} \cdot (X^{\alpha})^\dagger
+ \sum_{\len(\alpha)\leqslant m} \big[(X^{\alpha})^\dagger, M_{\overline{a_{\alpha}}}\big]\\
& = \sum_{\len(\alpha)= m} M_{\overline{a_{\alpha}}} \cdot (X^{\alpha})^\dagger
+ \sum_{\len(\beta)< m} M_{b_{\beta}} \cdot X^{\beta}.
\end{align*}

In the last equation, we merge all the summands with degree less than $m$ together, denote it by $\sum_{\len(\beta)< m} M_{b_{\beta}}X^{\beta}$, with $b_\beta \in C^{\infty}_b(G)$.
Taking $(\cdot)^{\rm top}_g$ both sides at $g\in G$, we get
\begin{equation*}
(P^{\dag})_g^{\rm top} 
= \sum_{\len(\alpha)= m} \overline{a_{\alpha}}(g)(X^{\alpha})^\dagger
= \big(\sum_{\len(\alpha)= m} a_{\alpha}(g)X^{\alpha}\big)^\dagger
= (P_g^{\rm top})^{\dag}.
\end{equation*}
\end{proof}

A similar proof shows 
\begin{lemma}\label{lem-symb-composition}
	For any differential operator $P$ and $Q$, $(PQ)^{\rm top}_g = P^{\rm top}_g Q^{\rm top}_g$.
\end{lemma}
\begin{proof}
	Set $P=\sum_{\len(\alpha)\leqslant m_1} M_{a_{\alpha}}X^{\alpha}$, $Q=\sum_{\len(\beta)\leqslant m_2} M_{b_{\beta}}X^{\beta}$.
	Their product reads
	\begin{align}\label{eq-PQ-top-g-produc-pf-1}
		PQ & = \sum_{\substack{\len(\alpha)\leqslant m_1\\ \len(\beta)\leqslant m_2}} 
			M_{a_{\alpha}}X^{\alpha}M_{b_{\beta}}X^{\beta}\nonumber \\
		& = \sum_{\substack{\len(\alpha)= m_1\\ \len(\beta)= m_2}}
			(M_{a_{\alpha}}M_{b_{\beta}}X^{\alpha}X^{\beta} + M_{a_{\alpha}}[X^{\alpha},M_{b_{\beta}}]X^{\beta})\nonumber \\
		& \quad	+ \sum_{\substack{\len(\alpha)+ \len(\beta)\\ < m_1 + m_2}}
			M_{a_{\alpha}}X^{\alpha}M_{b_{\beta}}X^{\beta}\nonumber  \\
		& = \sum_{\substack{\len(\alpha)= m_1\\ \len(\beta)= m_2}}
			M_{a_{\alpha}}M_{b_{\beta}}X^{\alpha}X^{\beta}
			+ \sum_{\len(\gamma)< m_1+m_2} M_{c_{\gamma}}X^{\gamma}.
	\end{align}
	For the last equation, recall that taking the commutator of a left-invariant differential operator with a multiplication operator lowers the degree of differential operator, 
	thus summands $M_{a_{\alpha}}[X^{\alpha},M_{b_{\beta}}]X^{\beta}$ are of degree at most $m_1 + m_2 - 1$. Merge all the summands with degree less than $m_1 + m_2$ together, denote it by $\sum_{\len(\gamma)< m_1+m_2} M_{c_{\gamma}}X^{\gamma}$, with $c_\gamma \in C^{\infty}_b(G)$.
	
	For both sides of \eqref{eq-PQ-top-g-produc-pf-1}, we take the top degree part and evaluate coefficients at point $g\in G$.
	\begin{multline*}
		(PQ)^{\rm top}_g 
		= \sum_{\substack{\len(\alpha)= m_1\\ \len(\beta)= m_2}}
			a_{\alpha}(g) b_{\beta}(g) X^{\alpha}X^{\beta}\\
		= (\sum_{\len(\alpha)= m_1} a_{\alpha}(g) X^{\alpha})
		\cdot (\sum_{ \len(\beta)= m_2} b_{\beta}(g)X^{\beta})
		= P^{\rm top}_g Q^{\rm top}_g.
	\end{multline*}
\end{proof}

Operation $(\cdot)^{\rm top}_g$ plays a similar role as taking principle symbol in pseudodifferential calculus. Compare Lemma \ref{lem-symb-composition} with \cite[Theorem 3.4]{Shubin-psido-2001} in Euclidean setting.

Often, a differential operator on $\mathbb{R}^d$ is said to be elliptic if the top degree component of its symbol is invertible except at zero. However, such a definition lacks a uniformity which is very important in the non-compact setting. We recall here the definition of  uniformly Rockland  stated in the introduction. For the special case $G=\mathbb{R}^d,$ this notion corresponds to a uniformly elliptic operator.

\begin{definition}\label{definition of ellipticity_1} Let $P$ be a differential operator of order $m.$ We say that $P$ is   uniformly Rockland  if there exists a constant $c_P>0$ such that
$$\|P_g^{{\rm top}}u\|_{L_2(G)}\geqslant c_P\|(-\Delta)^{\frac{m}{2{v}}} u\|_{L_2(G)},\quad u\in\Sc(G),\quad g\in G.$$
\end{definition}

We may think of $g\mapsto P_g$ as a bounded $\mathbb{C}^M$-valued function on $G$ (here, $M$ is the number of linearly independent $X^{\alpha},$ ${\rm len}(\alpha)\leqslant m$). It follows that there exists an extension of the function $g\mapsto P_g$ to the Stone-\v{C}ech compactification $\beta G$ of $G.$ 

\begin{lemma}\label{equivalent description of ellipticity} A differential operator $P$ is  uniformly Rockland  if and only if $\pi(P_g^{{\rm top}})$ is injective on $H^{\infty}_{\pi}$ for every $g\in\beta G$ and for every $\pi\in\widehat{G}.$ Here,
$$H^{\infty}_{\pi}=\bigcap_{\alpha}{\rm dom}(\pi(X^{\alpha}))$$
is the subspace of smooth vectors in the representation $(\pi,H_{\pi}).$
\end{lemma}
\begin{proof} If $P$ is elliptic, then
$$\|P_g^{{\rm top}}u\|_{L_2(G)}\geqslant c_P\|(-\Delta)^{\frac{m}{2{v}}} u\|_{L_2(G)},\quad u\in\Sc(G),\quad g\in \beta G.$$
{This is because of Definition \ref{definition of ellipticity_1} and $G$ is dense in $\beta G$.}
In other words,
$$\langle u,\Big((P_g^{{\rm top}})^{\dagger}P_g^{{\rm top}}-c_P^2(-\Delta)^{{\frac{m}{v}}}\Big)u\rangle\geqslant0,\quad u\in\Sc(G),\quad g\in \beta G.$$
By Theorem \ref{thm-abstract-rockland-general-g}, for every $\pi\in\widehat{G},$
$$\langle \xi,\pi\Big((P_g^{{\rm top}})^{\dagger}P_g^{{\rm top}}-c_P^2(-\Delta)^{{\frac{m}{v}}}\Big)\xi\rangle\geqslant0,\quad \xi\in H^{\infty}_{\pi},\quad g\in \beta G.$$
Equivalently, for every $\pi\in\widehat{G},$
$$\|\pi(P_g^{{\rm top}})\xi\|_{H_{\pi}}\geqslant c_P\|\pi((-\Delta)^{\frac{m}{2{v}}})\xi\|_{H_{\pi}},\quad \xi\in H_{\pi}^{\infty},\quad g\in\beta G.$$
This immediately yields the injectivity of $\pi(P_g^{{\rm top}})$ for every $g\in\beta G$ and for every $\pi\in\widehat{G}.$

Conversely, suppose that $\pi(P_g^{{\rm top}})$ is injective on $H^{\infty}_{\pi}$ for every $g\in\beta G$ and for every $\pi\in\widehat{G}.$ By \cite[Theorem 2.1 and Remark 2.2]{HelfferNourrigat1985}, for every $g\in\beta G$ there exists a constant $c_g>0$ such that
$$\|\pi(P_g^{{\rm top}})\xi\|_{H_{\pi}}\geqslant c_g\|\pi((-\Delta)^{\frac{m}{2{v}}})\xi\|_{H_{\pi}},\quad \xi\in H^{\infty}_{\pi},\quad \pi\in\widehat{G}.$$
By continuity,
$$\|\pi(P_g^{{\rm top}})\xi\|_{H_{\pi}}\geqslant c_g\|\pi((-\Delta)^{\frac{m}{2{v}}})\xi\|_{H_{\pi}},\quad \xi\in{\rm dom}(\pi((-\Delta)^{\frac{m}{2{v}}})),\quad \pi\in\widehat{G}.$$
{Since} the constant $c_g$ does not depend on $\pi\in\widehat{G},$ it follows from Plancherel theorem that
$$\|P_g^{{\rm top}}u\|_{L_2(G)}\geqslant c_g\|(-\Delta)^{\frac{m}{2{v}}}u\|_{L_2(G)},\quad u\in {\rm dom}((-\Delta)^{\frac{m}{2{v}}}).$$
In particular,
$$\|P_g^{{\rm top}}u\|_{L_2(G)}\geqslant c_g\|(-\Delta)^{\frac{m}{2{v}}}u\|_{L_2(G)},\quad u\in\Sc(G).$$

Set
$$f(g)=\inf\Big\{\frac{\|P_g^{{\rm top}}u\|_{L_2(G)}}{\|(-\Delta)^{\frac{m}{2{v}}}u\|_{L_2(G)}}:\ 0\not\equiv u\in\Sc(G)\Big\},\quad g\in\beta G.$$
If $g_1,g_2\in\beta G,$ then 
\begin{align*}
    \|P_{g_1}^{{\rm top}}u\|_{L_2(G)} &\geqslant\|P_{g_2}^{{\rm top}}u\|_{L_2(G)}-\|P_{g_1}^{{\rm top}}u-P_{g_2}^{{\rm top}}u\|_{L_2(G)}\\
                                    &\geqslant f(g_2)\|(-\Delta)^{\frac{m}{2{v}}}u\|_{L_2(G)}-\|P_{g_1}^{{\rm top}}-P_{g_2}^{{\rm top}}\|_{\mathring{W}^m_2(G)\to L_2(G)}\|(-\Delta)^{\frac{m}{2{v}}}u\|_{L_2(G)}.
\end{align*}
Thus,
$$f(g_1)\geqslant f(g_2)-\|P_{g_1}^{{\rm top}}-P_{g_2}^{{\rm top}}\|_{\mathring{W}^m_2(G)\to L_2(G)}.$$
Swapping $g_1$ and $g_2,$ we obtain
$$f(g_2)\geqslant f(g_1)-\|P_{g_1}^{{\rm top}}-P_{g_2}^{{\rm top}}\|_{\mathring{W}^m_2(G)\to L_2(G)}.$$
Thus,
$$|f(g_1)-f(g_2)|\leqslant \|P_{g_1}^{{\rm top}}-P_{g_2}^{{\rm top}}\|_{\mathring{W}^m_2(G)\to L_2(G)}.$$
In other words, the function $g\to f(g)$ is continuous on $\beta G.$ Since $f(g)\geqslant c_g>0$ for every $g\in\beta G,$ it follows from the compactness of $\beta G$ that $\inf_{g\in\beta G}f(g)>0.$ Appealing to the definition of $f,$ we conclude the ellipticity of $P.$
\end{proof}

\begin{lemma}\label{ellipticity of the product}
Let $P, Q$ be defined as in Definition \ref{def-differential-op}.
If $P$ and $Q$ are both uniformly Rockland, then so is $PQ$.
\end{lemma}
\begin{proof} If $P$ and $Q$ are uniformly Rockland, then, by Lemma \ref{equivalent description of ellipticity} $\pi(P_g^{{\rm top}})$ and $\pi(Q_g^{{\rm top}})$ are injective on $H^{\infty}_{\pi}$ for every $g\in\beta G$ and for every $\pi\in\widehat{G}.$ Hence, $\pi(P_g^{{\rm top}}Q_g^{{\rm top}})$ is injective on $H^{\infty}_{\pi}$ for every $g\in\beta G$ and for every $\pi\in\widehat{G}.$ Clearly, $P_g^{{\rm top}}Q_g^{{\rm top}}=(PQ)_g^{{\rm top}}$ for every $g\in\beta G.$ Thus, $\pi((PQ)_g^{{\rm top}})$ is injective on $H^{\infty}_{\pi}$ for every $g\in\beta G$ and for every $\pi\in\widehat{G}.$ Again appealing to Lemma \ref{equivalent description of ellipticity}, we infer that $PQ$ is uniformly Rockland.
\end{proof}

\begin{lemma}\label{lem-P-elliptic-P*P}
	Let $P$ be defined as in Definition \ref{def-differential-op}.
	$P$ is uniformly Rockland if and only if  $P^{\dagger}P$ is uniformly Rockland.
\end{lemma}
\begin{proof}
	Notice that, for any constant coefficient differential operator $A \in \Uc(\gf)$, for any given $\pi\in \widehat{G}$, 
	\begin{equation*}
		\ker(\pi(A^{\dagger}A)) \cap H^{\infty}_{\pi} = \ker(\pi(A)) \cap H^{\infty}_{\pi}.
	\end{equation*} 
	Indeed, for any $u\in \ker(\pi(A^{\dagger}A))\cap H^{\infty}_{\pi} $, 
	\begin{equation*}
		\Vert \pi(A)u\Vert^2_{H_{\pi}} = \langle \pi(A)^{\dagger}\pi(A)u,u \rangle_{H_{\pi}}
		= \langle \pi(A^{\dagger}A)u,u \rangle_{H_{\pi}}= 0.
	\end{equation*}
	This gives $\ker(\pi(A^{\dagger}A)) \cap H^{\infty}_{\pi} \subseteq \ker(\pi(A)) \cap H^{\infty}_{\pi}$. Since $\pi(A^{\dagger}A)= \pi(A^{\dagger})\pi(A)$, the other side is trivial.
	
	Lemma \ref{top dag commute} and \ref{lem-symb-composition} give us $(P^{\rm top}_g)^{\dagger}P^{\rm top}_g = (P^{\dagger})^{\rm top}_g P^{\rm top}_g = (P^{\dagger}P)^{\rm top}_g$.
	Thus, for any $\pi\in \widehat{G}$ and $g\in G$, $\pi((P^{\dagger}P)^{\rm top}_g)=\pi((P^{\rm top}_g)^{\dagger}P^{\rm top}_g)$ is injective on $H^{\infty}_{\pi}$ if and only if $\pi(P^{\rm top}_g)$ is injective on $H^{\infty}_{\pi}$.
	Appealing to Lemma \ref{equivalent description of ellipticity}, we infer that $P$ is uniformly Rockland if and only if $P^{\dagger}P$ is.
\end{proof}

\subsection{Forward elliptic estimate}\label{sub-sec-forward-elliptic}

In this section, we show that ellipticity as in Definition \ref{definition of ellipticity} implies the following elliptic estimate. 

\begin{theorem}\label{forward elliptic estimate} Let $P=P^{\dagger}$ be an  uniformly Rockland  order $m$ differential operator on $G.$ There exist $R_P,  \widetilde{c_P}\in(0,\infty)$ such that
$$\|(\widetilde{P}+ic)u\|_{L_2(G)}\geqslant \widetilde{c_P}\|u\|_{W^m_2(G)},\quad u\in \Sc(G),\quad c\in\mathbb{R},\quad |c|\geqslant R_P.$$
\end{theorem}

Throughout this section, we use $c_P$ to refer to the constant from Definition \ref{definition of ellipticity_1} and $\widetilde{c_P}$, $R_P$ from Theorem \ref{forward elliptic estimate}.

Theorem \ref{forward elliptic estimate} can be extended to $W^{m}_2(G)$.

\begin{corollary}\label{forward elliptic estimate 2} Let $P=P^{\dagger}$ be an  uniformly Rockland  order $m$ differential operator on $G.$ There exist $R_P, \widetilde{c_P}\in(0,\infty)$ such that
$$\|(\widetilde{P}+ic)u\|_{L_2(G)}\geqslant \widetilde{c_P}\|u\|_{W^m_2(G)},\quad u\in W^m_2(G),\quad c\in\mathbb{R},\quad |c|\geqslant R_P.$$
\end{corollary}
\begin{proof}
For any $u'\in W^m_2(G)$, choose $u \in \Sc(G)$, such that $\Vert u' - u\Vert_{W^m_2(G)}\leqslant  \epsilon $ \cite[Theorem 4.5]{Folland1975}. By triangle inequality and Theorem \ref{forward elliptic estimate}, we have
\begin{multline}\label{eq-schwartz-to-sobolev-cor-1}
\Vert (\widetilde{P}+i c)(u-u') \Vert_{L^2(G)} 
+\Vert (\widetilde{P}+i c)u' \Vert_{L^2(G)}\\
\geqslant \Vert (\widetilde{P}+i c)u \Vert_{L^2(G)}
\geqslant \widetilde{c_P}\Vert u \Vert_{W^m_2(G)}  \\
\geqslant \widetilde{c_P}(\Vert u' \Vert_{W^m_2(G)} - \Vert u'-u \Vert_{W^m_2(G)}) \\
\geqslant \widetilde{c_P}\Vert u' \Vert_{W^m_2(G)} - \epsilon \widetilde{c_P}.
\end{multline}
By Lemma \ref{differential_operators_are_bounded}, 
\begin{multline}\label{eq-schwartz-to-sobolev-cor-2}
\Vert (\widetilde{P}+i c)(u-u') \Vert_{L^2(G)} 
\leqslant  \Vert \widetilde{P}(u-u') \Vert_{L^2(G)} + |c|\Vert u-u' \Vert_{L^2(G)}\\
\leqslant  \|P\|_{W^m_2(G)\to L_2(G)}\Vert u-u' \Vert_{W^m_2(G)} + |c| \Vert u-u' \Vert_{L^2(G)}\\
\leqslant  (\|P\|_{W^m_2(G)\to L_2(G)} + |c|)\epsilon.
\end{multline}
Combining \eqref{eq-schwartz-to-sobolev-cor-1} and \eqref{eq-schwartz-to-sobolev-cor-2}, we get
\begin{equation*}
\Vert (\widetilde{P}+i c)u' \Vert_{L^2(G)} \geqslant \widetilde{c_P}\Vert u' \Vert_{W^m_2(G)} - \epsilon (\widetilde{c_P}+ |c| +\|P\|_{W^m_2(G)\to L_2(G)}).
\end{equation*}
Letting $\epsilon \to 0$, we obtain the result.
\end{proof}

Now we prove Theorem \ref{forward elliptic estimate}. First, we "freeze" the coefficients of the differential operator $P$. That is, for constant coefficient operator $P_g$, we prove the forward elliptic estimates up to a lower order term.
\begin{lemma}\label{fee constant coefficients} If $P=P^{\dagger}$ is an  uniformly Rockland  order $m$ differential operator on $G,$  then there exists a constant $c_{P,1}$, such that
$$\|(P_g+ic)u\|_{L_2(G)}\geqslant c_P\|u\|_{W^m_2(G)}-c_{P,1} \|u\|_{W^{m-1}_2(G)},\quad u\in \Sc(G),\quad g\in G,\quad c\in\mathbb{R},$$
here, $c_P$ is the constant in Definition \ref{definition of ellipticity_1}.
\end{lemma}
\begin{proof} By Lemma \ref{top dag commute} and the assumption, $P_g^{{\rm top}}=(P^{\dagger})_g^{{\rm top}} =(P_g^{{\rm top}})^{\dagger},$ it follows that
\begin{align}\label{eq-lemma-foraward-1}
\|(P_g^{{\rm top}}+ic)u\|_{L_2(G)}& =\big(\|P_g^{{\rm top}}u\|_{L_2(G)}^2+|c|^2\|u\|_{L_2(G)}^2\big)^{\frac12}\nonumber \\
& \geqslant\|P_g^{{\rm top}}u\|_{L_2(G)}\geqslant c_P\|u\|_{\mathring{W}^m_2(G)}.
\end{align}
The last inequality is the Definition \ref{definition of ellipticity_1} of ellipticity of $P$.

Since $P_g-P_g^{{\rm top}}$ has order at most $m-1$, and coefficients $a_\alpha$ are bounded, it follows that
$$c_P'\stackrel{def}{=}\sup_{g\in G}\|P_g-P_g^{{\rm top}}\|_{W^{m-1}_2(G)\to L_2(G)}<\infty.$$
We clearly have
\begin{equation}\label{eq-lemma-foraward-2}
\|(P_g-P_g^{{\rm top}})u\|_{L_2(G)}\leqslant c_P'\|u\|_{W^{m-1}_2(G)}.
\end{equation}
By triangle inequality, \eqref{eq-lemma-foraward-1} and \eqref{eq-lemma-foraward-2}, we have
\begin{align}\label{eq-lemma-foraward-3}
\|(P_g+ic)u\|_{L_2(G)}
&\geqslant\|(P_g^{{\rm top}}+ic)u\|_{L_2(G)}-\|(P_g-P_g^{{\rm top}})u\|_{L_2(G)} \nonumber \\
&\geqslant c_P\|u\|_{\mathring{W}^m_2(G)}-c_P'\|u\|_{W^{m-1}_2(G)}.
\end{align}
Recall from the definition of Sobolev norms \eqref{eq-sobolev-norm} and \eqref{eq-sobolev-norm-homogeneous}, we have 
$$\|u\|_{\mathring{W}^m_2(G)}\geqslant \|u\|_{W^m_2(G)}-c_m\|u\|_{W^{m-1}_2(G)},$$
Substitute into \eqref{eq-lemma-foraward-3}, we have
\begin{equation*}
\|(P_g+ic)u\|_{L_2(G)} \geqslant c_P\|u\|_{W^m_2(G)} - (c_P \cdot c_m + c'_P) \|u\|_{W^{m-1}_2(G)}.
\end{equation*}
\end{proof}

Now we "unfreeze" the coefficients, but only let them vary within a small set, and prove the same estimate. {Recall Definition \ref{C_infty_definition} and Definition \ref{def-differential-op}, all coefficients of a differential operator and their derivatives are assumed to be uniformly bounded over $G.$ Thus, the operator will preserve the same property when we "unfreeze" the coefficients.} 
\begin{lemma}\label{fee small diameter} 
If $P=P^{\dagger}$ is an  uniformly Rockland  order $m$ differential operator on $G,$ then there exists $\epsilon_P>0$, such that for the same constants $c_p$, $c_{P,1}$ as in Lemma  \ref{fee constant coefficients}, we have
$$\|(P+ic)u\|_{L_2(G)}\geqslant\frac12c_P\|u\|_{W^m_2(G)}-c_{P,1}\|u\|_{W^{m-1}_2(G)},\quad g\in G,\quad c\in\mathbb{R},$$
for every $u\in\Sc(G)$ with ${\rm diam}({\rm supp}(u))\leqslant\epsilon_P.$ Here, the diameter is understood with respect to any translation-invariant metric on $G$. 
\end{lemma}
\begin{proof} 
Recall $P$ is given by the formula \eqref{eq-P-expression}. For each $X^{\alpha}$ appearing in $P,$ we have $\len(\alpha)\leqslant m.$ We denote $c_{\alpha}=\|X^{\alpha}\|_{W^m_2(G)\to L_2(G)}.$
By Mean Value Theorem \cite[Proposition 3.1.46.]{FischerRuzhansky2016}, there exist a constant $C$, such that for any $g_1,g_2\in G$ with $\dist(g_1,g_2)<1$ (recall Definition \ref{def-homo-metric}), for any coefficient function $a_{\alpha}$ in $P$, we have
\begin{equation*}
	|a_{\alpha}(g_1)-a_{\alpha}(g_2)|
	\leqslant C \Vert a_{\alpha} \Vert_{s, b} \cdot d(g_1, g_2).
\end{equation*}
Here, $s$ is the highest degree of homogeneous vectors, as in Definition \ref{def-stratfied-g}. 
Since for all $\alpha$, $\Vert a_{\alpha} \Vert_{s, b}$ is always finite, recall Definition \ref{def-differential-op}, there exists a small $\epsilon_P >0 $ such that
$$\sum_{\len(\alpha)\leqslant m}c_{\alpha}|a_{\alpha}(g_1)-a_{\alpha}(g_2)|\leqslant\frac12c_P,
\quad \text{for all}\ g_1, g_2 \in G, \ d(g_1,g_2)\leqslant \epsilon_P .$$
Here, $c_P$ is the uniformly Rockland constant given by Definition \ref{definition of ellipticity_1}.

Fix $u\in \Sc(G)$ such that ${\rm diam}({\rm supp}(u))\leqslant\epsilon_P.$ Also, fix some $g\in{\rm supp}(u).$ Triangle inequality says
$$\|(P+ic)u\|_{L_2(G)}\geqslant\|(P_g+ic)u\|_{L_2(G)}-\|(P-P_g)u\|_{L_2(G)}.$$
Clearly,
\begin{align*}
\|(P-P_g)u\|_{L_2(G)}
&\leqslant\sum_{\len(\alpha)\leqslant m}\|M_{a_{\alpha}-a_{\alpha}(g)}X^{\alpha}u\|_{L_2(G)}\\
&\leqslant\sum_{\len(\alpha)\leqslant m}\|a_{\alpha}-a_{\alpha}(g)\|_{L_{\infty}({\rm supp}(u))}\|X^{\alpha}u\|_{L_2(G)}\\
&\leqslant\sum_{\len(\alpha)\leqslant m}c_{\alpha}\|a_{\alpha}-a_{\alpha}(g)\|_{L_{\infty}({\rm supp}(u))}\cdot \|u\|_{W^m_2(G)}.
\end{align*}
The last inequality is due to \eqref{integer sobolev norm}.
By the choice of $\epsilon_P,$ we have
$$\|(P+ic)u\|_{L_2(G)}\geqslant\|(P_g+ic)u\|_{L_2(G)}-\frac12c_P\|u\|_{W^m_2(G)}.$$
The assertion follows now from Lemma \ref{fee constant coefficients}.
\end{proof}

The following lemma shows that local estimates of commutators can be promoted to global ones.
\begin{lemma}\label{lemma-local-global-estimates-commutator}
Let $P$ be a differential operator of order $m$ on $G.$ Let $\psi\in C^{\infty}_c(G)$ {be} $1$ in some neighbourhood of $1_G.$ Let $\Psi=(\psi_n)_{n\geqslant0}$ be the partition of unity defined in Construction \ref{partition_of_unity}. For any $s\in \R$, we have
$$\Big(\sum_{n\geqslant0}\|[P,M_{\psi_n}]u\|_{W^s_2(G)}^2\Big)^{\frac12}\leqslant c_{P,\Psi,s}\|u\|_{W^{s+m-1}_2(G)},\quad u\in\Sc(G).$$
\end{lemma}
\begin{proof} Fix a real-valued $\phi\in C^{\infty}_c(G)$ such that $\phi\psi=\psi.$ Let $\{g_n\}_{n\geqslant0}$ be a sequence given by Construction \ref{partition_of_unity}. Set
$$\theta(g)=(\sum_{n\geqslant0}\phi^2(g_n^{-1}g))^{\frac12}, \qquad g\in G,$$
which is the denominator in \eqref{eq-partition-of-unity}. As stated in Construction \ref{partition_of_unity}, the sum is actually finite. Define $\phi_n$ similarly to \eqref{eq-partition-of-unity}, that is, $\phi_n (g)= \phi(g_n^{-1}g)/\theta(g)$ for $g\in G,$ $n\geqslant0.$ We have $\phi_n\theta\psi_n=\psi_n$ for every $n\geqslant0.$ 
Thus, for $n\geqslant0$,
\begin{equation}\label{eq-local-global-estimates-commutator-1}
[P,M_{\psi_n}]=[P,M_{\psi_n}M_{\phi_n\theta}] 
= [P,M_{\psi_n}]M_{\phi_n\theta} + M_{\psi_n}[P,M_{\phi_n\theta}] 
\end{equation}

By the Leibniz rule, there are constants $c_{\alpha,\beta,\beta'}$, such that
\begin{equation}\label{eq-commutation-P-f}
[X^{\alpha}, M_{f}] = \sum_{\substack{{\rm len}(\beta)+{\rm len}(\beta')={\rm len}(\alpha)\\ {\rm len}(\beta)\neq0}}  c_{\alpha,\beta,\beta'} \cdot M_{X^{\beta}(f)} \cdot X^{\beta'},\quad f\in C^{\infty}_c(G).
\end{equation}
Setting $f=\phi_n \theta$ in \eqref{eq-commutation-P-f}, we obtain
\begin{multline*}
M_{\psi_n}[P,M_{\phi_n\theta}]
=\sum_{{\rm len}(\alpha)\leqslant m}M_{\psi_n a_{\alpha}}[X^{\alpha},M_f]\\
=\sum_{{\rm len}(\alpha)\leqslant m}\sum_{\substack{{\rm len}(\beta)\neq0\\ {\rm len}(\beta)+{\rm len}(\beta')={\rm len}(\alpha)}}  c_{\alpha,\beta,\beta'} M_{\psi_n a_{\alpha}} \cdot M_{X^{\beta}(\phi_n \theta)} \cdot X^{\beta'}
=0.
\end{multline*}
Here, the last equality holds due to the  $\phi_n \theta = 1$ on the support of $\psi_n.$ 
Thus \eqref{eq-local-global-estimates-commutator-1} becomes
\begin{equation*}
[P,M_{\psi_n}] =[P,M_{\psi_n}]M_{\phi_n\theta} .
\end{equation*}
Hence, for $u \in \Sc(G)$,
$$\|[P,M_{\psi_n}]u\|_{W^s_2(G)}\leqslant\|[P,M_{\psi_n}]\|_{W^{s+m-1}_2(G)\to W^s_2(G)}\|\phi_n\theta u\|_{W^{s+m-1}_2(G)},\quad n\geqslant0.$$
Thus,
\begin{multline}\label{eq-local-commutator-P-psi-estimate}
\Big(\sum_{n\geqslant0}\|[P,M_{\psi_n}]u\|_{W^s_2(G)}^2\Big)^{\frac12}\\
\leqslant\sup_{n\geqslant0}\|[P,M_{\psi_n}]\|_{W^{s+m-1}_2(G)\to W^s_2(G)}\cdot \Big(\sum_{n\geqslant0}\|\phi_n\theta u\|_{W^{s+m-1}_2(G)}^2\Big)^{\frac12}.
\end{multline}

Setting $f=\phi_n \theta$ in \eqref{eq-commutation-P-f}, we obtain
\begin{multline*}
[P,M_{\psi_n}]=\sum_{{\rm len}(\alpha)\leqslant m}M_{a_{\alpha}}[X^{\alpha},M_f]\\
=\sum_{{\rm len}(\alpha)\leqslant m}\sum_{\substack{{\rm len}(\beta)\neq0\\ {\rm len}(\beta)+{\rm len}(\beta')={\rm len}(\alpha)}}  c_{\alpha,\beta,\beta'} M_{a_{\alpha}}\cdot M_{X^{\beta}(\psi_n)} \cdot X^{\beta'}.
\end{multline*}
By Definition \ref{def-differential-op}, $a_{\alpha}\in C_b^{\infty}(G)$. By Lemma \ref{multiplication_lemma}, $\Vert M_{a_\alpha}\Vert_{W^s_2(G)\to W^s_2(G)}$ are bounded.
By Corollary \ref{cor-operator-estimate-partition-unity}, $\Vert M_{X^{\beta}(\psi_n)}\Vert_{W^s_2(G)\to W^s_2(G)}$ are bounded for each $\beta$, uniformly in $n\in \mathbb{N}.$ Together with Lemma \ref{differential_operators_are_bounded}, we have 
\begin{equation*}
\sup_{n\geqslant0}\|[P,M_{\psi_n}]\|_{W^{s+m-1}_2(G)\to W^s_2(G)}< \infty .
\end{equation*}

Applying Theorem \ref{sobolev localization theorem} and Corollary \ref{cor-operator-estimate-partition-unity} to \eqref{eq-local-commutator-P-psi-estimate}, we complete the proof.
\end{proof}

\begin{lemma}\label{fee final}
If $P=P^{\dagger}$ is an  uniformly Rockland  order $m$ differential operator on $G,$ then there exist constants $c_{P,2}$ and $c_{P,3}$, such that
$$\|(P+ic)u\|_{L_2(G)}\geqslant c_{P,2}\|u\|_{W^m_2(G)}-c_{P,3}\|u\|_{W^{m-1}_2(G)},\quad c\in\mathbb{R},\quad u\in\Sc(G).$$
\end{lemma}
\begin{proof} Let $\epsilon_P$ be as in Lemma \ref{fee small diameter}. Fix a real-valued $\psi\in C^{\infty}_c(G)$ supported in $B(1_G,\frac12\epsilon_P)$ and equal to $1$ in some neighbourhood of $1_G.$ Let $(\psi_n)_{n\geqslant0}$ be the partition of unity defined in Construction \ref{partition_of_unity}.

Using Theorem \ref{sobolev localization theorem} with $s=0,$ we localise the norm in $L_2(G)$ as follows:
\begin{equation}\label{eq-foraward-elliptic-pf-with-lower-1}
c_{0,\Psi}\|(P+ic)u\|_{L_2(G)}\geqslant \Big(\sum_{n\geqslant0}\|M_{\psi_n}(P+ic)u\|_{L_2(G)}^2\Big)^{\frac12}.
\end{equation}
By triangle inequality in $L_2(G),$
\begin{equation}\label{eq-foraward-elliptic-pf-with-lower-2}
\|M_{\psi_n}(P+ic)u\|_{L_2(G)}\geqslant \|(P+ic)M_{\psi_n}u\|_{L_2(G)}-\|[P,M_{\psi_n}]u\|_{L_2(G)}.
\end{equation}
Since $\psi_n u$ has small support, it follows from Lemma \ref{fee small diameter} that
$$\|(P+ic)M_{\psi_n}u\|_{L_2(G)} \geqslant \frac12c_P\|\psi_n u\|_{W^m_2(G)}-c_{P,1}\|\psi_n u\|_{W^{m-1}_2(G)}.$$
Substituting back into \eqref{eq-foraward-elliptic-pf-with-lower-2}, we have
$$\|M_{\psi_n}(P+ic)u\|_{L_2(G)}\geqslant \frac12c_P\|\psi_n u\|_{W^m_2(G)}-c_{P,1}\|\psi_n u\|_{W^{m-1}_2(G)}-\|[P,M_{\psi_n}]u\|_{L_2(G)}.$$
By triangle inequality in $l_2,$
\begin{multline}\label{eq-foraward-elliptic-pf-with-lower-3}
\Big(\sum_{n\geqslant0}\|M_{\psi_n}(P+ic)u\|_{L_2(G)}^2\Big)^{\frac12}\geqslant \frac12c_P\Big(\sum_{n\geqslant0}\|\psi_n u\|_{W^m_2(G)}^2\Big)^{\frac12}\\
-c_{P,1}\Big(\sum_{n\geqslant0}\|\psi_n u\|_{W^{m-1}_2(G)}^2\Big)^{\frac12}-\Big(\sum_{n\geqslant0}\|[P,M_{\psi_n}]u\|_{L_2(G)}^2\Big)^{\frac12}.
\end{multline}

Combining \eqref{eq-foraward-elliptic-pf-with-lower-1} and \eqref{eq-foraward-elliptic-pf-with-lower-3}, we obtain
\begin{multline*}
c_{0,\Psi}\|(P+ic)u\|_{L_2(G)}\geqslant \frac12c_P\Big(\sum_{n\geqslant0}\|\psi_n u\|_{W^m_2(G)}^2\Big)^{\frac12}\\
-c_{P,1}\Big(\sum_{n\geqslant0}\|\psi_n u\|_{W^{m-1}_2(G)}^2\Big)^{\frac12}-\Big(\sum_{n\geqslant0}\|[P,M_{\psi_n}]u\|_{L_2(G)}^2\Big)^{\frac12}.
\end{multline*}
The first two term are localised norms, they are estimated by Theorem \ref{sobolev localization theorem}. The last term is estimated by Lemma \ref{lemma-local-global-estimates-commutator}.
\end{proof}

\begin{proof}[Proof of Theorem \ref{forward elliptic estimate}] Since $P=P^{\dagger},$ it follows that
$$\|(P+ic)u\|_{L_2(G)}=\big(\|Pu\|_{L_2(G)}^2+|c|^2\|u\|_{L_2(G)}^2\big)^{\frac12}\geqslant|c|\|u\|_{L_2(G)},\quad u\in\Sc(G).$$
By Lemma \ref{fee final},
$$\|(P+ic)u\|_{L_2(G)}\geqslant c_{P,2}\|u\|_{W^m_2(G)}-c_{P,3}\|u\|_{W^{m-1}_2(G)}.$$
Summing these inequalities, we arrive at
$$\|(P+ic)u\|_{L_2(G)}\geqslant \frac12c_{P,2}\|u\|_{W^m_2(G)}-\frac12c_{P,3}\|u\|_{W^{m-1}_2(G)}+\frac12|c|\|u\|_{L_2(G)}.$$
By \eqref{eq-interpolation-norm}, we have for each $\delta > 0$,
$$\|u\|_{W^{m-1}_2(G)}\leqslant\|u\|_{W^m_2(G)}^{1-\frac1m}\|u\|_{L_2(G)}^{\frac1m}\leqslant \delta \|u\|_{W^m_2(G)}+\delta^{1-m}\|u\|_{L_2(G)}.$$
Thus,
$$\|(P+ic)u\|_{L_2(G)}\geqslant \frac12(c_{P,2}-\delta c_{P,3})\|u\|_{W^m_2(G)}+\frac12(|c|-\delta^{1-m}c_{P,3})\|u\|_{L_2(G)}.$$
Setting
$$\widetilde{c_P}=\frac13c_{P,2},\quad \delta =\frac{c_{P,2}}{3c_{P,3}},\quad R_P=\frac{3^{m-1}c_{P,3}^m}{c_{P,2}^{m-1}},$$
we complete the proof.
\end{proof}


\subsection{Backward elliptic estimate}\label{sub-sec-backward-elliptic}

Now we are going to prove backward elliptic estimate, similar to Theorem \ref{forward elliptic estimate}, but with negative Sobolev degree.

\begin{theorem}\label{backward elliptic estimate} Let $P=P^{\dagger}$ be an  uniformly Rockland  order $m$ differential operator on $G.$ There exist $R_P,\widetilde{c_P}\in(0,\infty)$ such that
$$\|(P+ic)u\|_{W^{-m}_2(G)}\geqslant \widetilde{c_P}\|u\|_{L_2(G)},\quad u\in\Sc(G),\quad c\in\mathbb{R},\quad |c|\geqslant R_P.$$
\end{theorem}

\begin{theorem}\label{th-Q-self-adjoint}
Let $Q=Q^{\dagger}$ be an order $m$ differential operator on $G$ with constant coefficients. If $Q$ is homogeneous and elliptic, then $Q$ is essentially self-adjoint
with the closure having domain $W^m_2(G).$
\end{theorem}
\begin{proof}
See \cite[Proposition 4.1.15.]{FischerRuzhansky2016}.
\end{proof}

Recall that we denote by $\widetilde{Q}$ the distributional extension and by $\overline{Q}$ the minimal closed  extension of $Q$ as an operator on $L_2(G)$. Let $Q$ satisfy the conditions of Theorem \ref{th-Q-self-adjoint}, so that the domain of $\overline{Q}$ is $W^m_2(G).$ It follows that $\widetilde{Q}|_{W^m_2(G)} = \overline{Q}$. Moreover, by Theorem \ref{th-Q-self-adjoint}, $Q$ is essentially self-adjoint, thus for all $c\neq 0$, $(Q+ic)^{-1}$ is well defined, and maps $L_2(G)$ to $W^m_2(G)$.
In order to lighten the notations, in Lemmas \ref{backward constant coef compute lemma} and \ref{simple duality lemma}, we write $Q$ instead of $\widetilde{Q}|_{W^m_2(G)}$ and $\overline{Q}$.

\begin{lemma}\label{backward constant coef compute lemma} Let $Q=Q^{\dagger}$ be an order $m$ differential operator on $G$ with constant coefficients. Suppose $Q$ is homogeneous and elliptic. For every $c\in\R$, $c\neq 0$, we have
\begin{enumerate}[{\rm (i)}]
\item $\|(Q-ic)^{-1}\|_{L_2(G)\to W^m_2(G)}\leqslant c_{Q,c}^{-1}.$
\item $\|(Q-ic)^{-1}\|_{W^m_2(G)\to W^m_2(G)}\leqslant c_{Q,1}^{-1}|c|^{-1}(1+\|Q\|_{W^m_2(G)\to L_2(G)}).$
\end{enumerate}
Here, $c_{Q,c} = c_m \min\{|c|,c_Q\}$, $c_Q$ is the constant in Definition \ref{definition of ellipticity_1}, and $c_m$ is a constant only depending on $m$.
\end{lemma}

\begin{proof} 
To see the first assertion, fix $h\in W^m_2(G).$ Since $Q$ is self-adjoint, it follows that
\begin{multline}\label{eq-norm-square-c}
\|(Q-ic)h\|_{L_2(G)}^2=\langle Qh-ich,Qh-ich\rangle \\
=\langle Qh,Qh\rangle+c^2\langle h,h\rangle=\|Qh\|_{L_2(G)}^2+c^2\|h\|_{L_2(G)}^2.
\end{multline}
Since $Q$ is  uniformly Rockland  and homogeneous, it follows from Definition \ref{definition of ellipticity_1} that there exists a constant $c_Q$, such that 
\begin{align}\label{eq-estimate-Q-inverse-norm_1}
\|(Q-ic)h\|_{L_2(G)}^2
& \geqslant c_Q^2\|(-\Delta)^{\frac{m}{2{v}}}h\|_{L_2(G)}^2+c^2\|h\|_{L_2(G)}^2 \nonumber  \\
& \geqslant \min\{|c|,c_Q\}^2\cdot \big(\|(-\Delta)^{\frac{m}{2{v}}}h\|_{L_2(G)}^2+\|h\|_{L_2(G)}^2\big)\nonumber \\
& \geqslant c_m^2 \min\{|c|,c_Q\}^2\|h\|_{W^m_2(G)}^2.
\end{align}
Here, $c_m$ is a constant only depends on the order $m$.

Let $u\in L_2(G)$ and let $h=(Q-ic)^{-1}u.$ We have $h\in{\rm dom}(Q)=W^m_2(G)$ and $h=(Q-ic)v.$ Substituting into \eqref{eq-estimate-Q-inverse-norm_1}, we have
$$\|(Q-ic)^{-1}u\|_{W^m_2(G)}=\|h\|_{W^m_2(G)}\leqslant c_m^{-1}\min\{|c|,c_Q\}^{-1}\|u\|_{L_2(G)}.$$
Setting $c_{Q,c} = c_m\min\{|c|,c_Q\},$ we obtain the first assertion.

Let us now prove the second assertion. Set $c=1$ in \eqref{eq-estimate-Q-inverse-norm_1}, for any $h\in W^m_2(G)$,
$$\|(Q-i)h\|_{L_2(G)}\geqslant c_1\min\{1,c_Q\}\|h\|_{W^m_2(G)}=c_{1,Q}\|h\|_{W^m_2(G)}.$$
Using the latter inequality with $h=(Q-ic)^{-1}u,$ $u \in W_2^m(G)$ (note that $h\in W_2^m(G)$ as well), we obtain
\begin{align*}
\|(Q-ic)^{-1}u\|_{W^m_2(G)}
& \leqslant  c_{1,Q}^{-1}\|(Q-i)(Q-ic)^{-1}u\|_{L_2(G)}\\
& \leqslant  c_{1,Q}^{-1}|c|^{-1}\|(Q-i)u\|_{L_2(G)}\\
& \leqslant  c_{1,Q}^{-1}|c|^{-1}(1+\|Q\|_{W^m_2(G)\to L_2(G)})\|u\|_{W^m_2(G)}.
\end{align*}
This yields the second assertion.
\end{proof}

We need the following estimates obtained by duality.

\begin{lemma}\label{simple duality lemma} Let $Q=Q^{\dagger}$ be an order $m$ differential operator on $G$ with constant coefficients. Suppose $Q$ is homogeneous and elliptic. For every $c\in\mathbb{R}$, $c\neq 0$,  we have
$$\|(Q+ic)u\|_{W^{-m}_2(G)}\geqslant c_{Q,c}\|u\|_{L_2(G)},\quad u\in\Sc(G),$$
$$\|(Q+ic)u\|_{W^{-m}_2(G)}\geqslant \frac{c_{Q,1}|c|}{1+\|Q\|_{W^m_2(G)\to L_2(G)}}\|u\|_{W^{-m}_2(G)},\quad u\in\Sc(G).$$
Here, the constants $c_{Q,c}$ and $c_{Q,1}$ are taken from Lemma  \ref{backward constant coef compute lemma}.
\end{lemma}
\begin{proof} Since $u\in\Sc(G),$ it follows that
$$u=(Q+ic)^{-1}((Q+ic)u).$$	
By Theorem \ref{th-Q-self-adjoint}, $Q: L_2(G) \to L_2(G)$ is self-adjoint with domain $W^m_2(G)$. For every $h\in L_2(G),$  we have
$$|\langle u,h\rangle|=|\langle (Q+ic)^{-1}(Q+ic)u,h\rangle|=|\langle (Q+ic)u,(Q-ic)^{-1}h\rangle|.$$

The latter inner product can be also viewed as pairing between $W^{-m}_2(G)$ in the first argument and $W^m_2(G)$ in the second argument. Thus,
\begin{equation}\label{eq-norm-negative-Sobolev-Qc}
|\langle u,h\rangle|\leqslant\|(Q+ic)u\|_{W^{-m}_2(G)}\|(Q-ic)^{-1}h\|_{W^m_2(G)}.
\end{equation}

Dividing by $\|h\|_{L_2(G)}$ and taking the supremum over $0\neq h\in L_2(G),$ we infer
$$\|u\|_{L_2(G)}\leqslant \|(Q-ic)^{-1}\|_{L_2(G)\to W^m_2(G)}\|(Q+ic)u\|_{W^{-m}_2(G)},\quad u\in\Sc(G).$$
The first assertion follows now from Lemma \ref{backward constant coef compute lemma}.

On the other hand, dividing \eqref{eq-norm-negative-Sobolev-Qc} by $\|h\|_{W^m_2(G)}$ and taking the supremum over $0\neq h\in W^m_2(G),$ we infer
$$\|u\|_{W^{-m}_2(G)}\leqslant \|(Q-ic)^{-1}\|_{W^m_2(G)\to W^m_2(G)}\|(Q+ic)u\|_{W^{-m}_2(G)},\quad u\in\Sc(G).$$
The second assertion follows now from Lemma \ref{backward constant coef compute lemma}.
\end{proof}

The following norm estimate concerning $P_g$ is uniform in $g\in G$.

\begin{lemma}\label{bee constant coefficients} Let $P=P^{\dagger}$ be an  uniformly Rockland  order $m$ differential operator on $G.$ There exist constants $c_{P,4}$ and $r_P > 0$ such that, for every $c\in\mathbb{R}$ with $|c|\geqslant r_P,$ we have
$$\|(P_g+ic)u\|_{W^{-m}_2(G)}\geqslant c_{P,4}\big(\|u\|_{L_2(G)}+|c|\|u\|_{W^{-m}_2(G)}\Big),\quad u\in\Sc(G),\quad g\in G.$$
\end{lemma}
\begin{proof} In Lemma \ref{simple duality lemma}, set $Q = P_g^{\rm top}$ and recall formula for $c_{P_g^{\rm top}, c}$ in Lemma \ref{backward constant coef compute lemma},
$$c_{P_{g}^{\rm top},c} = c_m \min\{|c|,c_{P_g^{\rm top}}\}. $$
Here $c_{P_g^{\rm top}}$ is the constant in Definition \ref{definition of ellipticity_1} corresponding to the operator the operator $P_g^{{\rm top}}.$ It is bounded from below by the constant $c_P$  in Definition \ref{definition of ellipticity_1} corresponding to the operator $P.$ Thus,
$$c_{P_{g}^{\rm top},c}\geqslant  c_m \min\{|c|,c_P\}.$$

Since $P_g^{{\rm top}}$ has order $m$, and since the coefficients $a_\alpha$ are bounded, it follows that
\begin{equation*}
\sup_{g\in G}\|P_g^{{\rm top}}\|_{W^m_2(G)\to L_2(G)}< \infty .
\end{equation*}
We denote 
\begin{equation*}
c_P'=\frac{c_m c_P}{1+\sup_{g\in G}\|P_g^{{\rm top}}\|_{W^m_2(G)\to L_2(G)}} > 0.
\end{equation*}
Using Lemma \ref{simple duality lemma}, we obtain for every $c\in\mathbb{R}$ with $|c|\geqslant c_P,$
$$\|(P_g^{{\rm top}}+ic)u\|_{W^{-m}_2(G)}\geqslant c_P'|c|\|u\|_{W^{-m}_2(G)},\quad \|(P_g^{{\rm top}}+ic)u\|_{W^{-m}_2(G)}\geqslant c_P'\|u\|_{L_2(G)}.$$

Recall that $P_g-P_g^{\rm top}$ has order $m-1$ with constant coefficients $a_\alpha(g).$ Taking into account that each $a_\alpha$ is bounded, we set
$$c_P''=\sup_{g\in G}\|P_g-P_g^{{\rm top}}\|_{W^{-1}_2(G)\to W^{-m}_2(G)}< \infty .$$

Clearly,
\begin{multline*}
    \|(P_g+ic)u\|_{W^{-m}_2(G)}\geqslant \|(P_g^{{\rm top}}+ic)u\|_{W^{-m}_2(G)}-\|(P_g-P_g^{{\rm top}})u\|_{W^{-m}_2(G)}\\
    \geqslant \|(P_g^{{\rm top}}+ic)u\|_{W^{-m}_2(G)}-c_P''\|u\|_{W^{-1}_2(G)}.
\end{multline*}
By \eqref{eq-interpolation-norm}, for any $\epsilon >0$, 
$$\|u\|_{W^{-1}_2(G)}\leqslant \|u\|_{L_2(G)}^{1-\frac1m}\|u\|_{W^{-m}_2(G)}^{\frac1m}\leqslant \epsilon\|u\|_{L_2(G)}+\epsilon^{1-m}\|u\|_{W^{-m}_2(G)}.$$
It follows that
\begin{align*}
& \|(P_g+ic)u\|_{W^{-m}_2(G)}\\
& \geqslant \frac12c_P'\|u\|_{L_2(G)}+\frac12c_P'|c|\|u\|_{W^{-m}_2(G)}-\epsilon c_P''\|u\|_{L_2(G)}-\epsilon^{1-m}c_P''\|u\|_{W^{-m}_2(G)}\\
& = (\frac12c_P'-\epsilon c_P'')\|u\|_{L_2(G)}+(\frac12c_P'|c|-\epsilon^{1-m}c_P'')\|u\|_{W^{-m}_2(G)}.
\end{align*}
Taking small enough $\epsilon$ and, after that, taking large enough $|c|,$ we complete the proof.
\end{proof}

\begin{lemma}\label{bee small diameter} 
If $P=P^{\dagger}$ is an  uniformly Rockland  order $m$ differential operator on $G,$ then there exists $\epsilon_P,r_P\in(0,\infty)$ such that, for every $c\in\mathbb{R}$ with $|c|\geqslant r_P,$ we have
$$\|(P+ic)u\|_{W^{-m}_2(G)}\geqslant\frac12c_{P,4}\Big(\|u\|_{L_2(G)}+|c|\|u\|_{W^{-m}_2(G)}\Big),\quad g\in G,$$
for every $u\in\Sc(G)$ with ${\rm diam}({\rm supp}(u))\leqslant\epsilon_P.$ Here, the diameter is understood with respect to any translation-invariant metric.
\end{lemma}
\begin{proof} The argument follows that in Lemma \ref{fee small diameter} {\it mutatis mutandi}.
\end{proof}

\begin{proof}[Proof of Theorem \ref{backward elliptic estimate}] Let $\epsilon_P$ be as in Lemma \ref{fee small diameter}. Fix a real-valued $\psi\in C^{\infty}_c(G)$ supported in $B(1_G,\frac12\epsilon_P)$ and equal to $1$ in some neighbourhood of $1_G.$ Let $(\psi_n)_{n\geqslant0}$ be the partition of unity defined in Construction \ref{partition_of_unity}.

We have
$$c_{-m,\Psi}\|(P+ic)u\|_{W^{-m}_2(G)}\geqslant \Big(\sum_{n\geqslant0}\|M_{\psi_n}(P+ic)u\|_{W^{-m}_2(G)}^2\Big)^{\frac12}.$$
By triangle inequality in $W^{-m}_2(G),$
$$\|M_{\psi_n}(P+ic)u\|_{W^{-m}_2(G)}\geqslant \|(P+ic)M_{\psi_n}u\|_{W^{-m}_2(G)}-\|[P,M_{\psi_n}]u\|_{W^{-m}_2(G)}.$$
By Lemma \ref{bee small diameter}, we have
\begin{multline*}
    \|M_{\psi_n}(P+ic)u\|_{W^{-m}_2(G)}\\
    \geqslant \frac12c_{P,4}\Big(\|\psi_n u\|_{L_2(G)}+|c|\|\psi_n u\|_{W^{-m}_2(G)}\Big)-\|[P,M_{\psi_n}]u\|_{W^{-m}_2(G)}.
\end{multline*}

By triangle inequality in $l_2,$
\begin{multline*}
    \Big(\sum_{n\geqslant0}\|M_{\psi_n}(P+ic)u\|_{W^{-m}_2(G)}^2\Big)^{\frac12}\\
    \geqslant \frac14c_{P,4}\Big(\sum_{n\geqslant0}\|\psi_n u\|_{L_2(G)}^2\Big)^{\frac12}+\frac14c_{P,4}|c|\Big(\sum_{n\geqslant0}\|\psi_n u\|_{W^{-m}_2(G)}^2\Big)^{\frac12}\\
-\Big(\sum_{n\geqslant0}\|[P,M_{\psi_n}]u\|_{W^{-m}_2(G)}^2\Big)^{\frac12}.
\end{multline*}

Thus,
\begin{multline*}
    c_{-m,\Psi}\|(P+ic)u\|_{W^{-m}_2(G)}\\
\geqslant \frac14c_{P,4}\Big(\sum_{n\geqslant0}\|\psi_n u\|_{L_2(G)}^2\Big)^{\frac12}+\frac14c_{P,4}|c|\Big(\sum_{n\geqslant0}\|\psi_n u\|_{W^{-m}_2(G)}^2\Big)^{\frac12}\\
-\Big(\sum_{n\geqslant0}\|[P,M_{\psi_n}]u\|_{W^{-m}_2(G)}^2\Big)^{\frac12}.
\end{multline*}

Using Lemma \ref{lemma-local-global-estimates-commutator} and Theorem \ref{sobolev localization theorem}, we obtain
\begin{multline*}
c_{-m,\Psi}\|(P+ic)u\|_{W^{-m}_2(G)}\\
\geqslant \frac14c_{P,4}c_{0,\Psi}^{-1}\|u\|_{L_2(G)}+\frac14c_{P,4}c_{-m,\Psi}^{-1}|c|\|u\|_{W^{-m}_2(G)}-c_{P,\Psi}\|u\|_{W^{-1}_2(G)}.
\end{multline*}
Recall complex interpolation \eqref{eq-interpolation-norm} and Young's inequality, for any $\delta >0$, we have
$$\|u\|_{W^{-1}_2(G)}\leqslant \|u\|_{L_2(G)}^{1-\frac1m}\|u\|_{W^{-m}_2(G)}^{\frac1m}\leqslant \delta \|u\|_{L_2(G)}+\delta^{1-m}\|u\|_{W^{-m}_2(G)}.$$
It follows that
\begin{multline*}
c_{-m,\Psi}\|(P+ic)u\|_{W^{-m}_2(G)}\\
\geqslant (\frac14c_{P,4}c_{0,\Psi}^{-1}-c_{P,\Psi}\delta)\|u\|_{L_2(G)}+(\frac14c_{P,4}c_{-m,\Psi}^{-1}|c|-c_{P,\Psi}\delta^{1-m})\|u\|_{W^{-m}_2(G)}.
\end{multline*}
Taking small enough $\delta$ and, after that, taking large enough $|c|,$ we complete the proof.
\end{proof}

\subsection{Elliptic regularity and self-adjointness}\label{sub-sec-self-adjoint}

\begin{lemma}\label{zimmer theorem} Let $P=P^{\dagger}$ be an  uniformly Rockland  order $m$ differential operator on $G.$ For every $c\in\mathbb{R}$ with sufficiently large $|c|,$ the mapping
$$\widetilde{P}+ic:W^m_2(G)\to L_2(G)$$
is an isomorphism.
\end{lemma}
\begin{proof} 
Let $c\in\mathbb{R}$ be such that $|c|$ is sufficiently large. It follows from Theorem \ref{forward elliptic estimate} that
$(\widetilde{P}+ic):W^m_2(G)\to L_2(G)$ is injective. It remains to prove surjectivity of $(\widetilde{P}+ic):W^m_2(G)\to L_2(G).$

Select $h \in L_2(G)$ such that
$$\langle (P+ic)u,h\rangle=0,\quad u\in\Sc(G).$$
Here, $\langle\cdot,\cdot\rangle$ can be viewed as a pairing between $\Sc(G)$ in the first argument and $\Sc'(G)$ in the second argument. {By assumption, $P=P^{\dagger}$, also recall from \eqref{def-P-dagger} and Remark \ref{rmk-P-P-dagg-P-tilde-bdd} that we have
    \begin{equation*}
        \langle u,(\widetilde{P}-ic)h\rangle
        = \langle (P^{\dagger}+ic) u,h\rangle 
        = \langle (P+ic)u,h\rangle
        = 0,
        \quad u\in\Sc(G).
    \end{equation*}
    }
In other words, $(\widetilde{P}-ic)h=0.$ By Theorem \ref{backward elliptic estimate}, $(\widetilde{P}-ic):L_2(G)\to W^{-m}_2(G)$ is an injection. Hence, $h=0.$ Therefore, $(\widetilde{P}+ic):W^m_2(G)\to L_2(G)$ is surjective.
\end{proof}

\begin{proof}[Proof of Theorem \ref{elliptic regularity theorem} \ref{erta}] By Lemma \ref{zimmer theorem}, for every $c\in\mathbb{R}$ with sufficiently large $|c|,$ the mapping $\widetilde{P}+ic:W^m_2(G)\to L_2(G)$ is an isomorphism. Let $w=(\widetilde{P}+ic)u\in L_2(G).$ There exists $h \in W^m_2(G)$ such that $w=(\widetilde{P}+ic)h.$ Thus, $(\widetilde{P}+ic)(u-h)=0.$ Since $u-h\in L_2(G)$ and since $\widetilde{P}+ic:L_2(G)\to W^{-m}_2(G)$ is an isomorphic embedding, it follows that $u-h=0.$ Since $h\in W^m_2(G),$ it follows that $u\in W^m_2(G).$
\end{proof}

\begin{proof}[Proof of Theorem \ref{elliptic regularity theorem} \ref{ertb}] Denote the restriction of $\widetilde{P}$ to $W^m_2(G)$ by $Q.$

It follows from Lemma \ref{zimmer theorem} that $Q+ic:W^m_2(G)\to L_2(G)$ is an isomorphism. In particular, $Q:W^m_2(G)\to L_2(G)$ is closed.

By definition, $Q^{\ast}$ is defined on the domain
\begin{align*}
\mathrm{dom}(Q^{\ast}):= \{h \in L_2(G):\ &\text{There exists }C_h>0\text{ such that for all }\\
&u\in W^m_2(G)\text{ we have }|\langle h,Qu\rangle|\leqslant C_h\|u\|\}
\end{align*}
and $Q^{\ast}h$ is the unique element of $L_2(G)$ such that
$$\langle Q^{\ast}h,u\rangle = \langle h,Qu\rangle,\quad u\in W^m_2(G).$$
Since $\Sc(G)\subset W^m_2(G),$ it follows that
$$\langle Q^{\ast}h,\phi\rangle = \langle h,P\phi\rangle,\quad \phi\in \Sc(G),\; h \in \mathrm{dom}(Q^{\ast}).$$
Recall that the action of $\widetilde{P^{\dagger}}$ on distributions is defined by
$$\langle \widetilde{P^{\dagger}}\omega,\phi\rangle=\langle \omega,P\phi\rangle,\quad \phi\in \Sc(G),\quad \omega \in \Sc'(G).$$
Thus,
$$\langle Q^{\ast}h,\phi\rangle = \langle \widetilde{P^{\dagger}}h,\phi\rangle,\quad \phi\in \Sc(G),\; h \in \mathrm{dom}(Q^{\ast}).$$
Hence,
$$Q^{\ast}h=\widetilde{P^{\dagger}}h,\quad h \in \mathrm{dom}(Q^{\ast}).$$
Since $Q^{\ast}h\in L_2(G)$ for every $h \in \mathrm{dom}(Q^{\ast}),$ it follows that $\widetilde{P^{\dagger}}h\in L_2(G)$ for every $h \in \mathrm{dom}(Q^{\ast}).$ By assumption, $P^{\dagger}=P.$ Hence, $\widetilde{P}h\in L_2(G)$ for every $h \in \mathrm{dom}(Q^{\ast}).$  By Theorem \ref{elliptic regularity theorem} \ref{erta}, $h\in W^m_2(G)$ for every $h \in \mathrm{dom}(Q^{\ast}).$

It follows from the preceding paragraph that ${\rm dom}(Q^{\ast})\subset W^m_2(G).$ Since $P^{\dagger}=P,$ it follows that $Q^{\ast}\subset Q$ and $Q\subset Q^{\ast}.$ Thus, $Q^{\ast}=Q.$
\end{proof}

\subsection{General elliptic estimates}

Theorem \ref{forward elliptic estimate} implies, for an  uniformly Rockland  order $m$ differential operator $P=P^{\dagger},$ is elliptic, then for sufficiently large $|c|$ we have
$$\|h\|_{W^m_2(G)}\lesssim_P \|(P+ic)h\|_{L_2(G)},\quad h\in \Sc(G).$$
Similarly, Theorem \ref{backward elliptic estimate} implies that for sufficiently large $|c|,$
$$\|h\|_{L_2(G)}\lesssim_P \|(P+ic)h\|_{W^{-m}_2(G)},\quad h\in \Sc(G).$$
We now deduce the following generalisation.

\begin{theorem}\label{general elliptic estimate theorem} 
Let $P=P^{\dagger}$ be an  uniformly Rockland  differential operator of order $m.$ For every $s\in \mathbb{R},$ there exist constants $c_{P,s}>0$ and $R_{P,s}$ such that for all real $|c|>R_{P,s}$ we have
$$c_{P,s}\|h\|_{W^{s+m}_2(G)}\leqslant \|(P+ic)h\|_{W^s_2(G)},\quad h\in \Sc(G).$$
\end{theorem}

\begin{lemma}\label{gee first lemma} Let $P=P^{\dagger}$ be an  uniformly Rockland  differential operator of order $m.$
\begin{enumerate}[{\rm (i)}]
\item\label{geefla} $$(P+i)^{-1}:W^{mk}_2(G)\to W^{m(k+1)}_2(G),\quad {k\in \Z^{+}},$$
is an isomorphism of Hilbert spaces.
\item\label{geeflb} $$(P+i)^{-1}:\bigcap_{k\geqslant0}W^{mk}_2(G)\to \bigcap_{k\geqslant0}W^{mk}_2(G)$$
is a bijection.
\end{enumerate}
\end{lemma}
\begin{proof} For every self-adjoint operator $A:{\rm dom}(A)\to H,$ let us make the set ${\rm dom}(A^k)$ a Hilbert space by equipping it with the natural norm $u\to\|(A^k+i)u\|_H.$ By the self-adjoint functional calculus, the mapping
$$(A+i)^{-1}:{\rm dom}(A^k)\to {\rm dom}(A^{k+1})$$
is an isomorphism of Hilbert spaces. In particular, 
$$(A+i)^{-1}:\bigcap_{k\geqslant0}{\rm dom}(A^k)\to \bigcap_{k\geqslant0}{\rm dom}(A^k)$$
is a bijection.

If the operator $P$ satisfies the assumptions of the theorem, then, by Lemma \ref{ellipticity of the product}, so does the operator $P^k.$ By Theorem \ref{elliptic regularity theorem}, ${\rm dom}(P^k)=W^{mk}_2(G)$ with equivalent norms. The assertion follows immediately from the preceding paragraph.
\end{proof}

\begin{lemma}\label{gee second lemma} Let $P=P^{\dagger}$ be an  uniformly Rockland  differential operator of order $m.$ For every $s\geqslant0,$ there exists a constant $c_{P,s}>0$ such that
$$\|h\|_{W^{s+m}_2(G)}\leqslant c_{P,s}\|(P+i)h\|_{W^s_2(G)},\quad h\in \Sc(G).$$
\end{lemma}
\begin{proof} By Lemma \ref{gee first lemma} \ref{geefla} and complex interpolation, for every $s\geqslant0,$ $$(P+i)^{-1}:W^s_2(G)\to W^{s+m}_2(G)$$
is a bounded mapping. In particular,
$$\|(P+i)^{-1}h\|_{W^{s+m}_2(G)}\leqslant c_{P,s}\|h\|_{W^s_2(G)},\quad h\in \bigcap_{k\geqslant0}W^{mk}_2(G).$$
By Lemma \ref{gee first lemma} \ref{geeflb}, an arbitrary element $h \in \cap_{k\geqslant0}W^{mk}_2(G)$ can be represented as $h=(P+i)^{-1}u,$ $u\in\cap_{k\geqslant0}W^{mk}_2(G).$ Hence,
$$\|h\|_{W^{s+m}_2(G)}\leqslant c_{P,s}\|(P+i)h\|_{W^s_2(G)},\quad h\in \bigcap_{k\geqslant0}W^{mk}_2(G).$$
The assertion follows immediately.
\end{proof}

\begin{lemma}\label{gee third lemma} Let $P=P^{\dagger}$ be an  uniformly Rockland  differential operator of order $m.$ For every $s\leqslant -m,$
$$\|h\|_{W^{s+m}_2(G)}\leqslant c_{P,-s-m}\|(P+i)h\|_{W^s_2(G)},\quad h\in \Sc(G).$$
\end{lemma}
\begin{proof} By Lemma \ref{gee first lemma}\ref{geeflb}, an arbitrary element $w\in \cap_{k\geqslant0}W^{mk}_2(G)$ can be represented as $w=(P+i)u,$ $u\in\cap_{k\geqslant0}W^{mk}_2(G).$ Thus,
\begin{align}
    &\sup_{\substack{w\in\Sc(G)\\ \|w\|_{W^{-s-m}_2(G)}\leqslant1}}\|(P+i)^{-1}w\|_{W^{(-s-m)+m}_2(G)}\nonumber\\
    &\qquad \leqslant \sup_{\substack{w\in\cap_{k\geqslant0}W^{mk}_2(G)\\ \|w\|_{W^{-s-m}_2(G)}\leqslant1}}\|(P+i)^{-1}w\|_{W^{(-s-m)+m}_2(G)}\nonumber\\
    &\qquad =\sup_{\substack{u\in\cap_{k\geqslant0}W^{mk}_2(G)\\ \|(P+i)u\|_{W^{-s-m}_2(G)}\leqslant1}}\|u\|_{W^{(-s-m)+m}_2(G)}\nonumber\\
    &\qquad =\sup_{\substack{u\in\Sc(G)\\ \|(P+i)u\|_{W^{-s-m}_2(G)}\leqslant1}}\|u\|_{W^{(-s-m)+m}_2(G)}\stackrel{Lem.\ref{gee second lemma}}{\leqslant} c_{P,-s-m}.\label{gee_third_lemma_norm_comparison}
\end{align}
In the final step we have used $-s-m>0.$

We have
$$\|h\|_{W^{s+m}_2(G)}=\sup_{\substack{w\in\Sc(G)\\ \|w\|_{W^{-s-m}_2(G)}\leqslant1}}|\langle h, w\rangle|.$$
Since $P$ is symmetric,
$$\langle h,w\rangle=\langle (P-i)h,(P+i)^{-1}w\rangle,\quad v,w\in\Sc(G).$$
Thus,
$$|\langle h, w\rangle|\leqslant \|(P-i)h\|_{W^s_2(G)}\|(P+i)^{-1}w\|_{W^{-s}_2(G)},\quad v,w\in\Sc(G).$$
It follows that
$$\|h\|_{W^{s+m}_2(G)}\leqslant \|(P-i)h\|_{W^s_2(G)}\cdot\sup_{\substack{w\in\Sc(G)\\ \|w\|_{W^{-s-m}_2(G)}\leqslant1}}\|(P+i)^{-1}w\|_{W^{(-s-m)+m}_2(G)}.$$
The assertion follows now from \eqref{gee_third_lemma_norm_comparison}.
\end{proof}

\begin{lemma}\label{similar to zimmer lemma} Let $P=P^{\dagger}$ be an  uniformly Rockland  order $m$ differential operator on $G.$ For every $c\in\mathbb{R}$ with sufficiently large $|c|,$ the mapping
$$\widetilde{P}+ic:L_2(G)\to W^{-m}_2(G)$$
is an isomorphism.
\end{lemma}
\begin{proof} 
Let $c\in\mathbb{R}$ be such that $|c|$ is sufficiently large. It follows from Theorem \ref{backward elliptic estimate} that
$(\widetilde{P}+ic):L_2(G)\to W^{-m}_2(G)$ is injective. It remains to prove surjectivity of $(\widetilde{P}+ic):L_2(G)\to W^{-m}_2(G).$

Select $w\in W^{-m}_2(G)$ such that
$$\langle (\widetilde{P}+ic)u,w\rangle_{W^{-m}_2(G)}=0,\quad u\in L_2(G).$$
In particular,
$$\langle (P+ic)u,w\rangle_{W^{-m}_2(G)}=0,\quad u\in \Sc(G).$$
By definition, for $u, w\in \Sc(G)$,
$$\langle (P+ic)u,w\rangle_{W^{-m}_2(G)}=\langle (P+ic)u,(1-\Delta)^{-\frac{m}{{v}}} w\rangle.$$
Since $(1-\Delta)^{-\frac{m}{{v}} } w\in W^m_2(G)$ and since $P=P^{\dagger},$ it follows that
$$\langle u,(P-ic)\big((1-\Delta)^{-\frac{m}{{v}}} w\big)\rangle_{L_2(G)}=0,\quad u\in \Sc(G).$$
Hence,
$$(P-ic)\big((1-\Delta)^{-\frac{m}{{v}}} w\big)=0.$$
It follows from Theorem \ref{forward elliptic estimate 2} that $(1-\Delta)^{-\frac{m}{{v}}} w=0.$ This implies $w=0.$ Therefore, $(\widetilde{P}+ic):L_2(G)\to W^{-m}_2(G)$ is surjective.
\end{proof}

\begin{lemma}\label{gee fourth lemma} Let $P=P^{\dagger}$ be an  uniformly Rockland  differential operator of order $m.$ For every $s\in[-m,0]$ and for every $c\in\mathbb{R}$ with sufficiently large $|c|,$ we have
$$\|h\|_{W^{s+m}_2(G)}\leqslant c_{P,c}\|(P+ic)h\|_{W^s_2(G)},\quad h\in \Sc(G).$$
\end{lemma}
\begin{proof} By Lemma \ref{similar to zimmer lemma},
$$(\widetilde{P}+ic)^{-1}:W^{-m}_2(G)\to L_2(G)$$
is a well defined bounded mapping. By Lemma \ref{zimmer theorem},
$$(\widetilde{P}+ic)^{-1}:L_2(G)\to W^m_2(G)$$
is a well defined bounded mapping. By complex interpolation,
$$(\widetilde{P}+ic)^{-1}:W^s_2(G)\to W^{s+m}_2(G),\quad s\in [-m,0],$$
is a bounded mapping (and its norm is bounded uniformly in $s\in[-m,0]$). In particular,
$$\|(P+ic)^{-1}u\|_{W^{s+m}_2(G)}\leqslant c_{P,c}\|u\|_{W^s_2(G)},\quad u\in\bigcap_{k\geqslant0}W^{km}_2(G).$$
Applying Lemma \ref{gee first lemma} \ref{geeflb} to the operator $c^{-1}P,$ we represent an arbitrary element $h \in \bigcap_{k\geqslant0}W^{mk}_2(G)$ can be represented as $h =(P+ic)^{-1}u,$ $u\in\bigcap_{k\geqslant0}W^{mk}_2(G).$ Thus,
$$\|h\|_{W^{s+m}_2(G)}\leqslant c_{P,c}\|(P+ic)h\|_{W^s_2(G)},\quad v\in \bigcap_{k\geqslant0}W^{mk}_2(G).$$
The assertion follows immediately.
\end{proof}

\begin{proof}[Proof of Theorem \ref{general elliptic estimate theorem}] The assertion follows from Lemma \ref{gee second lemma} (for $s\geqslant0$) or Lemma \ref{gee third lemma} (for $s\leqslant -m$) or Lemma \ref{gee fourth lemma} (for $s\in[-m,0]$) and triangle inequality.
\end{proof}

Now we remove the condition $P^{\dagger}=P$ and obtain a more general statement, which is stated in the introduction as Theorem \ref{general elliptic estimate theorem 2}.
\begin{theorem}
Let $P$ be a differential operator of order $m$. If $P$ is  uniformly Rockland, then for every $s\in \mathbb{R}$, there exist constants $c_{P,s,1}, c_{P,s,2}>0$ such that 
\begin{equation*}
	c_{P,s,1}\|h\|_{W^{s+m}_2(G)}\leqslant \|Ph\|_{W^s_2(G)} + c_{P,s,2}\|h\|_{L_2(G)} ,\quad h\in \Sc(G).
\end{equation*}
\end{theorem}

\begin{proof}
	For uniformly Rockland operators $P$,  Lemma \ref{lem-P-elliptic-P*P} guarantees $P^{\dagger}P$ is also uniformly Rockland.
	For any fixed $s\in \R$, apply Theorem \ref{general elliptic estimate theorem} to $P^{\dagger}P$, there exist constants  $c_{P,s}>0$ and $R_{P,s}>0$ such that
	\begin{multline}\label{eq-gen2-1}
		c_{P,s}\|h\|_{W^{s+2m}_2(G)}
		\leqslant \|(P^{\dagger}P +iR_{P,s})h\|_{W^s_2(G)}\\
		\leqslant \Vert P^{\dagger}P h \Vert_{W^s_2(G)} + R_{P,s}\Vert h\Vert_{W^s_2(G)},
		\quad h\in \Sc(G).
	\end{multline}
	Recall Definition \ref{def-differential-op}, $P^\dagger = \sum_{\len(\alpha)\leqslant m} (X^{\alpha})^\dagger M_{\overline{a_{\alpha}}}$. By Lemma \ref{multiplication_lemma}, there exists a constant $C'_{P,s}>0$ such that
	\begin{equation}\label{eq-gen2-2}
		\Vert P^{\dagger} h\Vert_{W^s_2(G)} \leqslant C'_{P,s} \Vert h\Vert_{W^{s+m}_2(G)},
		\quad h\in \Sc(G).
	\end{equation}
	Combine \eqref{eq-gen2-1} and \eqref{eq-gen2-2}, we have
	\begin{equation*}
		c_{P,s}\|h\|_{W^{s+2m}_2(G)} \leqslant C'_{P,s} \Vert P h \Vert_{W^{s+m}_2(G)} + R_{P,s}\Vert h\Vert_{W^s_2(G)},
		\quad h\in \Sc(G).
	\end{equation*}
	Now, we divide constant $C'_{P,s}$ both sides. Since $s\in \R$ is arbitrary, we replace $s$ with $s-m$. Thus, there exist constants $c_{P,s,1}, c_{P,s,2}>0$ such that
	\begin{equation}\label{eq-pf-generalTh-s}
		c_{P,s,1}\|h\|_{W^{s+m}_2(G)} \leqslant  \Vert P h \Vert_{W^{s}_2(G)} + c_{P,s,2}\Vert h\Vert_{W^{s-m}_2(G)},
		\quad h\in \Sc(G).
	\end{equation}	
	When $s-m<0$, by property \eqref{item-Sobolev-emb} in Section \ref{sec-3} or \cite[Theorem 4.4.3]{FischerRuzhansky2016}, $\Vert h\Vert_{W^{s-m}_2(G)} \leqslant \Vert h\Vert_{L_2(G)}$, the assertion follows.
	Since order $m=0$ makes the assertion trivial, we assume $m > 0$ and $s-m>0$, recall interpolation of Sobolev space \eqref{eq-interpolation-norm} and Young's inequality, set $\theta = \frac{2m}{s+m}$, clearly $0< \theta < 1 $. For any $\epsilon >0$, the second term on the right hand side of \eqref{eq-pf-generalTh-s} become
	\begin{multline}\label{eq-pf-generalTh-interpolation}
		\Vert h\Vert_{W^{s-m}_2(G)} 
		\leqslant  \Vert h\Vert^{1-\theta}_{L_2(G)}\cdot \Vert h\Vert^{\theta}_{W^{s+m}_2(G)}\\
		\leqslant (1-\theta)(\frac{1}{\epsilon})^{\frac{1}{1-\theta}}\Vert h\Vert_{L_2(G)} 
			+ \theta \epsilon^{\frac{1}{\theta}} \Vert h\Vert_{W^{s+m}_2(G)}
	\end{multline}
	Substitute \eqref{eq-pf-generalTh-interpolation} back into \eqref{eq-pf-generalTh-s},	\begin{multline*}
		(c_{P,s,1}-c_{P,s,2}\theta \epsilon^{\frac{1}{\theta}}) \|h\|_{W^{s+m}_2(G)} \\
		\leqslant  \Vert P h \Vert_{W^{s}_2(G)} + c_{P,s,2} (1-\theta)(\frac{1}{\epsilon})^{\frac{1}{1-\theta}} \Vert h\Vert_{L_2(G)},
		\quad h\in \Sc(G).
	\end{multline*}
	Let $\epsilon$ be small enough, and rename constants,  we prove the assertion.
\end{proof}


\subsection{Example}\label{example 4 section}   

\begin{example}\label{example-Heins-op}
Let $\mathfrak{h}_1=\mathrm{span}\{X,Y,T\}$ be the three-dimensional Heisenberg algebra, and let $\mathbb{H}^1$ be the corresponding Heisenberg group. Let $f\in C^\infty_b(\mathbb{H}^1).$  Differential operator 
$$P=-X^2-Y^2+iM_fT$$
is  uniformly Rockland  if and only
$$\inf_{g\in\mathbb{H}^1}{\rm dist}(f(g),2\mathbb{Z}+1)>0.$$
\end{example}
\begin{proof} Consider irreducible representations $\pi_+$ and $\pi_-$ of $\mathbb{H}^1$ on $L_2(\mathbb{R})$ defined by the formulae
$$\pi_+(X)=ip,\quad \pi_+(Y)=iq,\quad \pi_+(T)=i,$$
$$\pi_-(X)=ip,\quad \pi_-(Y)=-iq,\quad \pi_-(T)=-i.$$

Recall Definition \ref{def-P-top}, we have
$$P_g=P_g^{{\rm top}} =-X^2-Y^2+if(g)T.$$
Thus,
$$\pi_{\pm}(P_g^{{\rm top}})=p^2+q^2\mp f(g).$$
If $P$ is elliptic, then
$$\sup_{g\in\mathbb{H}^1}\|(p^2+q^2\mp f(g))^{-1}\|_{\infty}<\infty.$$
Since ${\rm spec}(p^2+q^2)=2\mathbb{Z}_++1,$ it follows that
$$\sup_{g\in\mathbb{H}^1}\sup_{n\in 2\mathbb{Z}_++1}|(n- f(g))^{-1}|<\infty,\quad \sup_{g\in\mathbb{H}^1}\sup_{n\in 2\mathbb{Z}_++1}|(n+f(g))^{-1}|<\infty.$$
In other words,
$$\inf_{g\in\mathbb{H}^1}{\rm dist}(f(g),2\mathbb{Z}_++1)>0,\quad \inf_{g\in\mathbb{H}^1}{\rm dist}(-f(g),2\mathbb{Z}+1)>0.$$
This proves necessity.

Proof of sufficiency is similar.
\end{proof}


\appendix

\section{Higson's positivity criterion}\label{higson_appendix}
The results in this section were originally proved by Higson for the upcoming work \cite{HLMSZ}. With his kind permission, we reproduce them here.
\subsection{Positivity for differential operators on nilpotent Lie groups}

The following theorem characterises the positivity of right-invariant differential operators on a nilpotent Lie group in terms of their images under unitary representations.
\begin{theorem} \label{thm-abstract-rockland-general-g} Let $G$ be a nilpotent Lie group. If $D\in\mathcal{U}(\mathfrak{g})$ is such that
$$\langle u, Du \rangle_{L_2 (G)} \ge 0,\quad u\in\Sc(G),$$
then
$$\langle v, \pi(D) v \rangle_{H_\pi} \ge0,\quad v\in H^{\infty}_{\pi}$$
for every $\pi\in\widehat{G}.$ Here, $H^{\infty}_{\pi}$ denotes the set of all smooth vectors.
\end{theorem}
We state the theorem for nilpotent groups since that is our focus, but the proof of the theorem only uses the fact that $G$ is amenable.

Notice that the Plancherel theorem would imply the assertion for Plancherel almost every irreducible unitary representation $\pi\in\widehat{G}.$ Theorem \ref{thm-abstract-rockland-general-g} proves it for every $\pi\in\widehat{G}.$

\subsection{Preliminaries on the group $C^{\ast}$-algebra}
The space $L_1(G)$ is an algebra under right-convolution $\ast.$ If $\pi:G\to U(H_{\pi})$ is an irreducible strongly continuous unitary representation of $G,$ then there exists an algebra representation
$$V_{\pi}:(L_1(G),\ast)\to \mathcal{B}(H_{\pi})$$
defined by the $H_{\pi}$-valued Bochner integral
$$(V_{\pi}f)\xi = \int_{G} f(g)\pi(g)\xi\, dg,\quad \xi\in H_{\pi}.$$
By the triangle inequality, we have
$$\|V_{\pi}f\|_{B(H_{\pi})} \leqslant \|f\|_1.$$

For $f\in L_1(G),$ set 
$$f^{\#}(g):=\overline{f(g^{-1})},\quad g \in G.$$
We have
$$(V_{\pi}f)^{\ast}=V_{\pi}(f^{\#}),\quad f\in L_1(G).$$
Hence, $V_{\pi}$ is a $\ast$-representation.

The {\it full group $C^{\ast}$-algebra} $C^{\ast}(G)$ is defined as the completion of $L_1(G)$ with respect to the norm
$$\|f\|_{C^{\ast}(G)}\stackrel{def}{=}\sup_{\pi\in\widehat{G}} \|V_{\pi}f\|_{B(H_{\pi})}.$$

The {\it left-regular representation}, $(\lambda,L_2(G))$ is defined by
$$(\lambda(g)u)(h) = f(g^{-1}h),\quad g,h\in G,\quad u \in L_2(G).$$
The {\it reduced group $C^{\ast}$-algebra} $C^{\ast}_r(G)$ is defined as the completion of $L_1(G)$ with respect to the norm
$$\|f\|_{C^{\ast}_r(G)}\stackrel{def}{=}\|V_{\lambda}f\|_{B(L_2(G))}.$$

\begin{theorem}\label{group_positivity_theorem} Let $G$ be a nilpotent Lie group and let $f \in L_1(G).$ If $V_{\lambda}f\in B(L_2(G))$ is positive, then so is $V_{\pi}f\in B(H_{\pi})$ for every $\pi\in\widehat{G}.$
\end{theorem}
\begin{proof} Since $G$ is nilpotent (and is, therefore, amenable), it follows that $C^{\ast}_r(G)$ coincides with $C^{\ast}(G)$ \cite[Theorem VII.2.5]{Davidson}.
If $V_{\lambda}f$ is positive in $L_2(G),$ then $f$ is positive in $C^{\ast}_r(G).$ Hence, $f$ is positive in $C^{\ast}(G).$ The full group $C^*$-algebra
$C^{\ast}(G)$ is naturally isomorphic to the closure of the set
$$\Big\{\bigoplus_{\pi\in\widehat{G}}V_{\pi}f:\ f\in L_1(G)\Big\}\subset B\Big(\bigoplus_{\pi\in\widehat{G}}H_{\pi}\Big)$$
in the uniform norm. Since $f$ is positive in $C^{\ast}(G)$ and since positivity is preserved by $\ast$-isomorphisms of $C^{\ast}$-algebras, it follows that
$$\bigoplus_{\pi\in\widehat{G}}V_{\pi}f\geqslant 0.$$
This completes the proof.
\end{proof}

\subsection{Proof of Theorem \ref{thm-abstract-rockland-general-g}}
Given a unitary representation $(\pi,H_{\pi})$ of a Lie group $G,$ the subspace $H^{\infty}_{\pi}$ of smooth vectors is the joint domain of every $\pi(P),$ where $P\in \mathcal{U}(\mathfrak{g}).$ This is equipped with a
canonical Fr\'echet topology.
\begin{definition} The Fr\'echet topology on $H_{\pi}^{\infty}$ is given by the family of seminorms
$$\|\xi\|_{\pi,n} := \Big(\sum_{{\rm len}(\alpha)\leqslant n} \|\pi(X^{\alpha})\xi\|_{H_{\pi}}^2\Big)^{\frac12},\quad \xi\in H_{\pi}^\infty,\quad n\in\mathbb{Z}_+.$$
\end{definition}

\begin{lemma}\label{poor_mans_dixmier_malliavin} Let $(\phi_{\epsilon})_{\epsilon>0}$ be an approximate identity on $G.$ For every unitary representation $\pi\colon G \to U(H_\pi)$ and for every $\xi\in H^{\infty}_{\pi},$ we have $(V_{\pi}\phi_{\epsilon})\xi\to\xi$ as $\epsilon\downarrow0$ in the Fr\'echet topology on $H_{\pi}^{\infty}.$
\end{lemma}
\begin{proof}
Since the map $g\mapsto \pi(g)\xi$ is continuous from $G$ to $H_{\pi}$ for $\xi \in H_{\pi}^\infty,$ we have
$$\lim_{g\to 1_G} \|\pi(g)\xi-\xi\|_{\pi,n}=0,\quad n\geqslant 1.$$
Let $\{\phi_{\varepsilon}\}_{\varepsilon>0}\subset C^\infty_c(G)$ be an approximate identity, where $\phi_{\varepsilon}$ is compactly supported. Given $\xi \in H_{\pi}^\infty,$ as $\varepsilon\to 0$ we have
$$\|\xi-(V_{\pi}\phi_{\varepsilon})\xi\|_{\pi,n} \leqslant \int_{G} |\phi_{\varepsilon}(g)|\|\xi-\pi(g)\xi\|_{\pi,n}\,dg\to 0.$$
So that $\pi(\phi_{\varepsilon})\xi$ converges to $\xi$ in the topology of $H^{\infty}_{\pi}.$
\end{proof}

\begin{lemma}\label{nigel computaitonal lemma} In the assumptions of Theorem \ref{thm-abstract-rockland-general-g}, we have
$$V_{\lambda}(u^{\#}\ast Du)\geqslant0,\quad u\in\Sc(G).$$
\end{lemma}
\begin{proof} By assumption, 
$$\langle w, Dw\rangle \geqslant 0,\quad w\in\Sc(G).$$
Replacing $w$ with $u\ast v$ and using the equality $D(u\ast v) =Du\ast v,$ we obtain
$$\langle u\ast v, Du\ast v\rangle \geqslant 0,\quad u,v\in\Sc(G).$$
Taking into account that
$$u\ast v=(V_{\lambda}u)v,\quad Du\ast v=(V_{\lambda}(Du))v,$$
we write
$$\langle (V_{\lambda}u)v,(V_{\lambda}(Du))v\rangle\geqslant 0,\quad u,v\in\Sc(G).$$
Hence,
$$\langle v,(V_{\lambda}u)^{\ast}(V_{\lambda}(Du))v\rangle\geqslant 0,\quad u,v\in\Sc(G).$$
Since 
$$(V_{\lambda}u)^{\ast}(V_{\lambda}(Du))=V_{\lambda}(u^{\#})(V_{\lambda}(Du))=V_{\lambda}(u^{\#}\ast Du),$$
it follows that
$$\langle v,V_{\lambda}(u^{\#}\ast Du)v\rangle \geqslant 0,\quad u,v\in\Sc(G).$$
By continuity,
$$\langle v,V_{\lambda}(u^{\#}\ast Du)v\rangle \geqslant 0,\quad u\in\Sc(G),\quad v\in L_2(G).$$
In other words,
$$V_{\lambda}(u^{\#}\ast Du)\geqslant0,\quad u\in\Sc(G).$$
\end{proof}

\begin{proof}[Proof of Theorem \ref{thm-abstract-rockland-general-g}] Let $\pi\in\widehat{G}.$ By Lemma \ref{nigel computaitonal lemma} and Theorem \ref{group_positivity_theorem},
$$(V_{\pi}u)^{\ast}\cdot V_{\pi}(Du)=V_{\pi}(u^{\#}\ast Du)\geqslant0,\quad u\in\Sc(G).$$
In other words,
$$\langle (V_{\pi}u)\xi,(V_{\pi}(Du))\xi\rangle\geqslant0,\quad u\in\Sc(G),\quad \xi\in H_{\pi}.$$
In particular,
$$\langle (V_{\pi}u)\xi,(V_{\pi}(Du))\xi\rangle\geqslant0,\quad u\in\Sc(G),\quad \xi\in H_{\pi}^{\infty}.$$
Since
$$(V_{\pi}(Du))\xi=\pi(D)((V_{\pi}u)\xi),\quad u\in\Sc(G),\quad \xi\in H^{\infty}_{\pi},$$
it follows that
$$\langle (V_{\pi}u)\xi,\pi(D)((V_{\pi}u)\xi)\rangle\geqslant0,\quad u\in\Sc(G),\quad \xi\in H_{\pi}^{\infty}.$$

Let $(\phi_{\epsilon})_{\epsilon>0}$ be an approximate identity on $G$. By the preceding paragraph,
$$\langle (V_{\pi}\phi_{\epsilon})\xi,\pi(D)((V_{\pi}\phi_{\epsilon})\xi)\rangle\geqslant0,\quad u\in\Sc(G),\quad \xi\in H_{\pi}^{\infty}.$$
By Lemma \ref{poor_mans_dixmier_malliavin},
$$\langle (V_{\pi}\phi_{\epsilon})\xi,\pi(D)((V_{\pi}\phi_{\epsilon})\xi)\rangle\to \langle \xi,\pi(D)\xi\rangle,\quad \epsilon\downarrow 0.$$
This completes the proof.
\end{proof}

\bibliographystyle{alpha}
\bibliography{bibliography.bib}

\end{document}